\documentclass[11pt]{article}

\usepackage{amssymb,amsmath,amscd,amsthm,xy,enumitem,mathrsfs,stmaryrd,titlesec}
\xyoption{all}
\titleformat{\subsubsection}[runin]{\normalfont\bfseries}{\thesubsubsection}{0.5em}{}[.]

\newtheorem*{thm*}{Theorem}
\newtheorem{thm}{Theorem}[section]
\newtheorem{prp}[thm]{Proposition}
\newtheorem{lmm}[thm]{Lemma}

\theoremstyle{definition}

\theoremstyle{remark}

\numberwithin{equation}{section}

\def\lra{\longrightarrow}

\def\xlra#1{\xrightarrow{#1}}

\def\BE#1{\begin{equation}\label{#1}}
\def\EE{\end{equation}}

\def\lr#1{\langle#1\rangle}

\def\blr#1{\big\langle#1\big\rangle}
\def\dbsqbr#1{\llbracket{#1}\rrbracket}

\def\wt#1{\widetilde{#1}}
\def\ov#1{\overline{#1}}
\def\eref#1{(\ref{#1})}
\def\tn#1{\textnormal{#1}}
\def\sf#1{\textsf{#1}}

\def\Ga{\Gamma}

\def\al{\alpha}
\def\be{\beta}

\def\ga{\gamma}

\def\la{\lambda}

\def\th{\theta}

\def\vph{\varphi}

\def\C{\mathbb C}

\def\bE{\mathbb E}

\def\bF{\mathbb{F}}

\def\bI{\mathbb I}
\def\cI{\mathcal I}

\def\cN{\mathcal N}

\def\P{\mathbb P}
\def\R{\mathbb{R}}

\def\Z{\mathbb{Z}}

\def\fc{\mathfrak c}

\def\Aut{\tn{Aut}}
\def\au{\tn{au}}
\def\BL{\mathbb{BL}}
\def\BLau{\BL^{\!\!\fc_1}}
\def\BLauk{\BL^{\!\!\fc_1,k}}
\def\BLaua{\BL^{\!\!\fc_1,1}}
\def\BLaub{\BL^{\!\!\fc_1,2}}
\def\BLauba{\BL^{\!\!\fc_1,2;1}}
\def\BLC{\BL^{\!\!\C}}
\def\BLF{\BL^{\!\!\bF}}
\def\BLR{\BL^{\!\!\R}}
\def\nd{\tn{d}}

\def\End{\tn{End}}

\def\Hom{\tn{Hom}}
\def\id{\tn{id}}

\def\so{\tn{so}}
\def\SO{\tn{SO}}

\def\ver{\tn{ver}}

\def\i{\infty}
\def\prt{\partial}

\begin{document}

\title{Smooth Blowups: Global vs.~Local Perspectives}
\author{Aleksey Zinger\thanks{Partially supported by NSF grant DMS 2301493}}
\date{\today}
\maketitle

\begin{abstract}
\noindent   
We show that the global and local constructions of three types of blowup 
of a smooth manifold along a closed submanifold in differential topology are 
equivalent.
\end{abstract}

\tableofcontents

\section{Introduction}
\label{intro_sec}

\noindent
Blowups are surgery constructions that replace a closed submanifold~$Y$ in a smooth manifold~$X$
by a fiber bundle over~$Y$.
They play prominent roles in algebraic geometry, symplectic geometry, and differential topology.
In algebraic geometry, blowups are constructed using coordinate charts; see \cite[p603]{GH}.
In symplectic geometry, blowups are usually constructed via a tubular neighborhood identification;
see \cite[Section~8]{GS}.
In differential topology, they can be constructed using either coordinate charts,
as in the proof of \cite[Lemma~2.1]{AkKi85} and \cite[Section~3]{ArKa}, or
a tubular neighborhood identification, as at the bottom of \cite[p111]{AkKi81}.
The global approach via a tubular neighborhood identification
is quicker, less cumbersome, and more convenient 
for studying the changes in the topology under blowups, 
as in~\cite[Section~3]{Keel} and~\cite[Section~3.2]{RDMhomol}.
However, the local approach via coordinate charts
may be more suitable in settings with natural collections of 
coordinate charts, as is the case in~\cite[Sections~4,5]{RDMbl}, for example.
It also applies in the real analytic category and leads to functorial properties
for the three kinds of blowups in the spirit of \cite[Sections~4,5]{ArKa}.
In this paper, we show that the global and local approaches to blowup constructions
in differential topology are equivalent.\\

\noindent
Throughout this paper, $X$ denotes a smooth manifold and $Y\!\subset\!X$ 
a closed submanifold of (real or complex) codimension~$\fc$ 
($Y$ is closed as a topological subspace of~$X$, is without boundary, 
but may be non-compact if $X$ is not compact).
We consider three types of blowups of~$X$ along~$Y$,
\BE{blowdownmaps_e}\pi\!:\BLR_YX\lra X, \qquad \pi\!:\BLC_YX\lra X, 
\qquad\hbox{and}\qquad \pi\!:\BL_{\cN^{\fc_1}}X\lra X,\EE
which we call the \sf{real}, \sf{complex}, and \sf{$\cN^{\fc_1}$-augmented} 
(or just $\fc_1$-augmented) blowups, respectively.
The last two blowups involve a choice of complex structure on the normal bundle~$\cN_XY$
of~$Y$ in~$X$ and a subbundle \hbox{$\cN^{\fc_1}\!\subset\!\cN_XY$} of corank~$\fc_1\!<\!\fc$,
respectively.
The three maps in~\eref{blowdownmaps_e} are surjective, proper, and smooth.
The restrictions
\begin{gather}
\label{standdiff_e}
\pi\!:\BLR_YX\!-\!\pi^{-1}(Y)\lra X\!-\!Y, \qquad 
\pi\!:\BLC_YX\!-\!\pi^{-1}(Y)\lra X\!-\!Y,\\
\label{augdiff_e}
\hbox{and}\qquad 
\pi\!:\BL_{\cN^{\fc_1}}X\!-\!\pi^{-1}(Y)\lra X\!-\!Y
\end{gather}
are diffeomorphisms.
The local constructions of the complex and augmented blowups depend 
on the choice of a compatible collection of local charts for~$X$ that cover~$Y$.
Such a collection determines a complex structure on~$\cN_XY$ in the complex case 
and a distinguished subbundle $\cN^{\fc_1}\!\subset\!\cN_XY$ 
in the augmented case.
The global constructions of these two blowups depend 
on the choice of a tubular neighborhood identification.
The global construction of an augmented blowup is actually two real blowups,
first along~$Y$ and then along $\R\P\cN^{\fc_1}\!\subset\!\R\P(\cN_XY)$,
followed by a real blowdown in a different direction.\\

\noindent
The restrictions of~$\pi$ to the \sf{exceptional loci}
$$\pi\!:\bE_Y^{\R}X\!\equiv\!\pi^{-1}(Y)\lra Y \qquad\hbox{and}\qquad
\pi\!:\bE_Y^{\C}X\!\equiv\!\pi^{-1}(Y)\lra Y$$
in the first two cases
are isomorphic to the real and complex projectivizations, 
$\R\P(\cN_XY)$ and $\C\P(\cN_XY)$, respectively, of~$\cN_XY$.
The normal bundles to the exceptional loci in these two cases are
isomorphic to the real and complex tautological line bundles:
\BE{cNEisom_e}\cN_{\BLR_YX}\bE_Y^{\R}X\approx\ga_{\cN_XY}^{\R}\lra\R\P(\cN_XY)
\quad\hbox{and}\quad
\cN_{\BLC_YX}\bE_Y^{\C}X\approx\ga_{\cN_XY}^{\C}\lra\C\P(\cN_XY)\,.\EE
For $k\!=\!0,1$, let \hbox{$\tau_Y^k\!\lra\!Y$} be the trivial rank~$k$ real line bundle.
In the last case in~\eref{blowdownmaps_e}, the exceptional locus
$$\bE_{\cN^{\fc_1}}X\equiv\pi^{-1}(Y)\subset\wt{X}$$ 
is the union of two closed submanifolds, $\bE^0_{\cN^{\fc_1}}X$ and~$\bE^-_{\cN^{\fc_1}}X$, so that 
$$\big(\bE^0_{\cN^{\fc_1}}X,\bE^0_{\cN^{\fc_1}}X\!\cap\!\bE^-_{\cN^{\fc_1}}X\big)\approx
\big(\R\P\big(\cN_YX/\cN^{\fc_1}\!\oplus\!\tau_Y^1\big),
\R\P\big(\cN_YX/\cN^{\fc_1}\!\oplus\!\tau_Y^0\big)\big)$$
as fiber bundle pairs over~$Y$, while 
$$\pi\!:\bE^-_{\cN^{\fc_1}}X\lra Y$$
factors through an $S^{\fc-\fc_1}$-fiber bundle 
$\bE^-_{\cN^{\fc_1}}X\!\lra\!\bE^0_{\cN^{\fc_1}}X\!\cap\!\bE^-_{\cN^{\fc_1}}X$.

\begin{thm*}\label{main_thm}
Let $X$ be a smooth manifold and $Y\!\subset\!X$ be a closed submanifold of codimension~$\fc$.
\begin{enumerate}[label=(\arabic*),leftmargin=*]

\item\label{RCthm_it} The global construction of real and complex blowups in
Section~\ref{glRCbl_subs} is equivalent to the local construction in Section~\ref{locRCbl_subs}.

\item\label{augthm_it} If in addition $\cN^{\fc_1}\!\subset\!\cN_XY$ is a subbundle
of corank~$\fc_1\!<\!\fc$,
the global construction of $\cN^{\fc_1}$-augmented blowup in
Section~\ref{glaugbl_subs} is equivalent to the local construction in Section~\ref{locaugbl_subs}.

\end{enumerate}

\end{thm*}

\noindent
It is straightforward to cut up a tubular neighborhood identification (with additional structure)
used to construct a blowup \hbox{$\pi\!:\wt{X}\!\lra\!X$}
as in Section~\ref{glRCbl_subs} or~\ref{glaugbl_subs} into a compatible collection
of coordinate charts (with additional structure) that determines 
a blowup \hbox{$\pi'\!:\wt{X}'\!\lra\!X$}
as in Section~\ref{locRCbl_subs} or~\ref{locaugbl_subs}, respectively.
Lemma~\ref{StanBl_lmm} or~\ref{AugBl_lmm}, as appropriate, then provides a diffeomorphism
\BE{wtPsidfn_e}\wt\Psi\!:\wt{X}\lra\wt{X}' \qquad\hbox{s.t.}\quad 
\pi\!=\!\pi'\!\circ\!\wt\Psi\!:\wt{X}\lra X\,,\EE
i.e.~the two blowups are \sf{isomorphic}.
It is more technical to assemble a compatible collection of coordinate charts used to construct
a blowup \hbox{$\pi'\!:\wt{X}'\!\lra\!X$} in Section~\ref{locRCbl_subs} or~\ref{locaugbl_subs}
into a tubular neighborhood identification
suitable for constructing a blowup \hbox{$\pi\!:\wt{X}\!\lra\!X$};
this is done via Proposition~\ref{Loc2GlRC_prp} or~\ref{Loc2Glaug_prp}, as appropriate.
A tubular neighborhood identification obtained in this way can be cut up into 
coordinate charts that are compatible with the original charts and thus determine
a blowup isomorphic to \hbox{$\pi'\!:\wt{X}'\!\lra\!X$}.
By Lemma~\ref{StanBl_lmm} or~\ref{AugBl_lmm},  the blowup determined by the new charts is also
isomorphic to \hbox{$\pi\!:\wt{X}\!\lra\!X$}. 
This establishes the theorem.
Sections~\ref{locRCbl_subs} and~\ref{locaugbl_subs} in the present paper
are essentially identical to the main parts of 
Sections~3.1 and~3.2, respectively, in~\cite{RDMbl};
we include them for the ease of~use.

\section{Real and complex blowups}
\label{RCbl_sec}

\noindent
For $m\!\in\!\Z^{\ge0}$, we define $[m]\!=\!\{1,\ldots,m\}$.
Let $\fc\!\in\!\Z^+$ and $\bF\!=\!\R,\C$.
We denote~by
$$\ga^{\bF}_{\fc}\equiv\big\{\!(\ell,v)\!\in\!\bF\P^{\fc-1}\!\times\!\bF^{\fc}\!:
v\!\in\!\ell\!\subset\!\bF^{\fc}\big\}$$
the \sf{tautological line bundle} over the $\bF$-projective space $\bF\P^{\fc-1}$.
For each $i\!\in\![\fc]$, let
\begin{gather}\label{vphfcidfn_e}
\wt\vph_{\fc;i}\!\equiv\!
\big(\wt\vph_{\fc;i;1},\ldots,\wt\vph_{\fc;i;\fc}\big)\!:
\wt{U}_{\fc;i}^{\bF}\!\equiv\!\big\{\!
\big([r_1,\ldots,r_{\fc}],v\big)\!\in\!\ga^{\bF}_{\fc}\!:
r_i\!\neq\!0\big\}\lra\bF^{\fc},\\
\notag
\wt\vph_{\fc;i;j}\big([r_1,\ldots,r_{\fc}],(v_1,\ldots,v_{\fc})\!\big)
=\begin{cases}r_j/r_i,&\hbox{if}~j\!\in\![\fc]\!-\!\{i\};\\
v_i,&\hbox{if}~j\!=\!i;
\end{cases}
\end{gather}
be the $i$-th standard coordinate chart on~$\ga^{\bF}_{\fc}$.\\

\noindent
For an $\bF$-vector bundle $\pi_{\cN}\!:\cN\!\lra\!Y$,
we denote by~$\bF\P\cN$ the $\bF$-projectivization of~$\cN$ and~by 
\BE{gacNbFdfn_e}\ga_{\cN}^{\bF}\equiv
\big\{\!(\ell,v)\!\in\!\bF\P\cN\!\times\!\cN\!:v\!\in\!\ell\big\}
\lra\bF\P\cN\EE
the $\bF$-tautological line bundle.
If the $\bF$-rank of~$\cN$ is~$r$, then $\bF\P\cN\!\lra\!Y$ is a smooth $\bF\P^{r-1}$-bundle.
We identify $\ga_{\cN}^{\bF}\!-\!\bF\P\cN$ with $\cN\!-\!Y$ via the projection 
to the second component; this is used in~\eref{BRREdiag_e}.

\subsection{Global construction}
\label{glRCbl_subs}

\noindent
Let \hbox{$\pi_{\cN_XY}\!:\cN_XY\!\lra\!Y$} be the normal bundle of~$Y$ in~$X$ and
$$\pi_Y^{\perp}\!: TX|_Y\lra\frac{TX|_Y}{TY}\!\equiv\!\cN_XY$$
be the quotient projection.
A \sf{tubular neighborhood identification} for $Y\!\subset\!X$ 
\sf{over an open subset \hbox{$U_Y\!\subset\!Y$}} is 
a pair $(W_Y,\Psi_Y)$ consisting of an open neighborhood~$W_Y\!\subset\!\cN_XY$ of~$U_Y$
and a diffeomorphism
\BE{tubneighbdfn_e}\Psi_Y\!:W_Y\lra U_X\EE
onto an open neighborhood $U_X\!\subset\!X$ of~$U_Y$
so that $Y\!\cap\!W_Y\!=\!U_Y$, $\Psi_Y|_{U_Y}\!=\!\id_{U_Y}$, and the homomorphism
$$T(\cN_XY)\big|_{U_Y} \xlra{\nd\Psi_Y} TX|_{U_Y}
\xlra{\pi_Y^{\perp}} \frac{TX|_Y}{TY}\big|_{U_Y}\!\equiv\!\cN_XY\big|_{U_Y}$$
restricts to the identity homomorphism on \hbox{$\cN_XY|_{U_Y}\!\subset\!T(\cN_XY)|_{U_Y}$}.
Such a diffeomorphism~$\Psi_Y$ exists by the Tubular Neighborhood Theorem.
We call a tubular neighborhood identification for $Y\!\subset\!X$ over $U_Y\!=\!Y$
simply a \sf{tubular neighborhood identification} for $Y\!\subset\!X$.\\

\noindent
If $\cN_XY$ is a complex vector bundle, let $\bF\!=\!\C$;
otherwise, let $\bF\!=\!\R$.
Let $\fc$ be the $\bF$-rank of~$\cN_XY$.
For $W\!\subset\!\cN_XY$, define
$$\wt{W}^{\bF}=\big\{\!(\ell,v)\!\in\!\ga_{\cN_XY}^{\bF}\!: v\!\in\!W\big\}.$$
Given a tubular neighborhood identification~$(W_Y,\Psi_Y)$ for $Y\!\subset\!X$, define
\begin{gather*}
\BLF_YX\equiv\big(\!(X\!-\!Y)\!\sqcup\!\wt{W}_Y^{\bF}\!\big)\!\big/\!\!\sim,
\quad \wt{W}_Y^{\bF}\!-\!\bF\P(\cN_XY)\ni(\ell,v)\sim\Psi_Y(v)\in U_X\!-\!Y\subset X\!-\!Y,\\
\bE_Y^{\bF}X\equiv \bF\P(\cN_XY)\subset \BLF_YX.
\end{gather*}
The \sf{$\bF$-blowup} $\BLF_YX$ of~$X$ along~$Y$ is a smooth manifold containing 
the \sf{exceptional locus}~$\bE_Y^{\bF}X$ as a closed submanifold
with normal bundle $\ga_{\cN_XY}^{\bF}$.
The \sf{blowdown~map}
$$\pi\!: \BLF_YX\lra X, \qquad
\pi(\wt{x})=\begin{cases}\wt{x},&\hbox{if}~\wt{x}\!\in\!X\!-\!Y;\\
\Psi_Y(v),&\hbox{if}~\wt{x}\!\equiv\!(\ell,v)\!\in\!\wt{W}_Y^{\bF};
\end{cases}$$
is well-defined, surjective, proper, and smooth. The restrictions
$$\pi\!:\BLF_YX\!-\!\bE_Y^{\bF}X\lra X\!-\!Y \qquad\hbox{and}\qquad
\pi\!:\bE_Y^{\bF}X\lra Y$$
are a diffeomorphism and a smooth $\bF\P^{\fc-1}$-bundle, respectively.
Shrinking the open subset \hbox{$W_Y\!\subset\!\cN_XY$} does not change the smooth isomorphism
class of the blowdown map~$\pi$.\\

\noindent
We call tubular neighborhood identifications $(W_1,\Psi_1)$ and~$(W_2,\Psi_2)$ 
for $Y\!\subset\!X$ over open subsets $U_1\!\subset\!Y$ and $U_2\!\subset\!Y$, respectively,
\sf{$\bF$-equivalent} if there exist an open neighborhood 
\hbox{$W_{\cap}\!\subset\!\Psi_2^{-1}(\Psi_1(W_1)\!)$} of $U_1\!\cap\!U_2$ and
\BE{TNIequiv_e}\begin{split}
&h\in\Ga\big(W_{\cap},\Hom_{\bF}\big(\pi_{\cN_XY}^*\cN_XY,
\Psi_2^*\Psi_1^{-1*}\pi_{\cN_XY}^*\cN_XY\big)\!\big)\\
&\quad\hbox{s.t.}\qquad \Psi_1^{-1}\big(\Psi_2(v)\!\big)=\big\{\!h(v)\!\big\}(v)
\quad\forall\,v\!\in\!W_{\cap}\,.
\end{split}\EE
Since $\Psi_1(U_1)\!=\!U_1$ and $\Psi_2(W_2\!-\!Y)\!\subset\!X\!-\!Y$,
$\{\!h(v)\!\}(v)\!\neq\!0$ for all $v\!\in\!W_{\cap}\!-\!Y$.
By the sentence containing~\eref{tubneighbdfn_e}, such a smooth section~$h$ must satisfy
\BE{TNIprop_e}\begin{split}
\Psi_1^{-1}\!\circ\!\Psi_2\big|_{Y\cap W_{\cap}}\!=\!\id_{Y\cap W_{\cap}}\!:
Y\!\cap\!W_{\cap}&\lra Y\!\cap\!W_{\cap} \qquad\hbox{and}\\
h\big|_{Y\cap W_{\cap}}\!=\!\id_{\cN_XY|_{Y\cap W_{\cap}}}\!:
\cN_XY|_{Y\cap W_{\cap}}&\lra\cN_XY|_{Y\cap W_{\cap}}\,.
\end{split}\EE
The above determines an equivalence relation on the collection of all 
tubular neighborhood identifications for $Y\!\subset\!X$ (over $U_Y\!=\!X$).
We show below that tubular neighborhood identifications 
for $Y\!\subset\!X$ that are $\bF$-equivalent determine isomorphic $\bF$-blowups of~$X$ along~$Y$.
Since any two tubular neighborhood identifications are $\R$-equivalent,
the $\R$-blowup of~$X$ along~$Y$ as constructed above does not depend on the choice of 
tubular neighborhood identification for $Y\!\subset\!X$.\\

\noindent
Suppose $(W_Y,\Psi_Y)$ and $(W_Y',\Psi_Y')$ are tubular neighborhood identifications for $Y\!\subset\!X$
so that \hbox{$\Psi_Y'(W_Y')\!\subset\!\Psi_Y(W_Y)$} and
$h$ is in~\eref{TNIequiv_e} with $\Psi_1\!=\!\Psi_Y$, $\Psi_2\!=\!\Psi_Y'$, and $W_{\cap}\!=\!W_Y'$.
Let
$$\wt{W}_Y^{\bF},\wt{W}_Y'^{\bF}\subset\ga_{\cN_XY}^{\bF}$$
be the open subsets determined by $W_Y,W_Y'$ and
$$\pi\!: \BLF_YX\lra X \qquad\hbox{and}\qquad \pi'\!: \big(\BLF_YX\big)'\lra X$$
be the corresponding blowdown maps.
By the sentence below~\eref{TNIequiv_e} and~\eref{TNIprop_e}, the~map
\BE{wtvphalalprdfn_e0}\wt\vph\!:\wt{W}_Y'^{\bF}\lra \wt{W}_Y^{\bF},  \qquad
\wt\vph(\ell,v)=\big(\!\{h(v)\}(\ell),\{h(v)\}(v)\!\big),\EE
is well-defined and smooth.
It induces a diffeomorphism
$$\wt\Psi\!:\big(\BLF_YX\big)'\lra \BLF_YX,\qquad
 \wt\Psi\big([\wt{x}]\big)=\begin{cases}[\wt{x}],&\hbox{if}~\wt{x}\!\in\!X\!-\!Y\;\\
\big[\wt\vph(\wt{x})\!\big],&\hbox{if}~\wt{x}\!\in\!\wt{W}_Y'^{\bF};
\end{cases}$$
so that $\pi'\!=\!\pi\!\circ\!\wt\Psi$.

\subsection{Local construction}
\label{locRCbl_subs}

\noindent
Suppose $Y\!\subset\!X$ is a closed submanifold of $\bF$-codimension~$\fc$ and real dimension~$m$.
We call a coordinate chart 
$$\vph\!\equiv\!(\vph_1,\ldots,\vph_{\fc+m})\!:U\!\lra\!\bF^{\fc}\!\times\!\R^m$$
on~$X$ a \sf{chart for~$Y$ in~$X$} if 
\BE{vphYcond_e}U\!\cap\!Y=\big\{x\!\in\!U\!:
\vph_1(x),\ldots,\vph_{\fc}(x)\!=\!0\big\}.\EE
For such a chart, the subspace
$$\BLF_Y\vph\equiv \big\{\!(\ell,x)\!\in\!\bF\P^{\fc-1}\!\times\!U\!:
\big(\vph_1(x),\ldots,\vph_{\fc}(x)\!\big)\!\in\!\ell\!\subset\!\bF^{\fc}\big\}$$
of $\bF\P^{\fc-1}\!\times\!U$ is a closed submanifold.
We denote by 
$\pi_{\bF\P^{\fc-1}},\pi_U\!:\BLF_Y\vph\!\lra\!\bF\P^{\fc-1},U$
the two projections.
The latter restricts to a diffeomorphism
$$\pi_U\!:\big\{\!(\ell,x)\!\in\!\bF\P^{\fc-1}\!\times\!U\!:x\!\not\in\!Y\big\}\lra U\!-\!Y.$$
For each $i\!\in\![\fc]$, the smooth map
\begin{gather}\label{BLFchart_e}
\wt\vph_i\!\equiv\!(\wt\vph_{i;1},\ldots,\wt\vph_{i;\fc+m})\!:
\BLF_{Y;i}\vph\!\equiv\!\big\{\!([r_1,\ldots,r_{\fc}],x)\!\in\!\BLF_Y\vph\!:
r_i\!\neq\!0\big\}
\lra\bF^{\fc}\!\times\!\R^m,\\
\notag
\wt\vph_{i;j}=\begin{cases}
\wt\vph_{\fc;i;j}\!\circ\!\pi_{\bF\P^{\fc-1}},
&\hbox{if}~j\!\in\![\fc]\!-\!\{i\};\\
\vph_j\!\circ\!\pi_U,
&\hbox{if}~j\!\in\!\{i\}\!\cup\!\big([\fc\!+\!m]\!-\![\fc]\big);
\end{cases}
\end{gather}
is a coordinate chart on~$\BLF_Y\vph$.\\

\noindent
Let $\{\vph_{\al}\!:U_{\al}\!\lra\!\bF^{\fc}\!\times\!\R^m\}_{\al\in\cI}$ be a collection
of charts for~$Y$ in~$X$.
For $\al,\al'\!\in\!\cI$, let
$$\vph_{\al\al'}\!\equiv\!\big(\vph_{\al\al';1},\ldots,\vph_{\al\al';\fc+m}\big)
\!\equiv\!\vph_{\al}\!\circ\!\vph_{\al'}^{-1}\!:
\vph_{\al'}\big(U_{\al}\!\cap\!U_{\al'}\big)\lra \vph_{\al}\big(U_{\al}\!\cap\!U_{\al'}\big)$$
be the overlap map between the charts $\vph_{\al}$ and $\vph_{\al'}$.
By~\eref{vphYcond_e},
\BE{Rslice_e}
\vph_{\al\al'}^{~-1}\big(0^{\fc}\!\times\R^m\big)
=\vph_{\al'}\big(U_{\al}\!\cap\!U_{\al'}\big)\!\cap\!\big(0^{\fc}\!\times\!\R^m\big).\EE
We call $\{\vph_{\al}\}_{\al\in\cI}$ an \sf{$\bF$-atlas for~$Y$ in~$X$}
if the domains~$U_{\al}$ of~$\vph_{\al}$ cover~$Y$ and
 for all $\al,\al'\!\in\cI$ there exists a smooth map 
\BE{halalprdfn_e}\begin{split}
&\hspace{1in}
h_{\al\al'}\!:\vph_{\al'}\big(U_{\al}\!\cap\!U_{\al'}\big)\lra\End_{\bF}(\bF^{\fc})
\qquad\hbox{s.t.}\\
&\big(\vph_{\al\al';1}(r,s),\ldots,\vph_{\al\al';\fc}(r,s)\!\big)=h_{\al\al'}(r,s)r
\quad\forall\,
(r,s)\!\in\!\vph_{\al'}(U_{\al}\!\cap\!U_{\al'})\!\subset\!\bF^{\fc}\!\times\!\R^m.
\end{split}\EE
By~\eref{Rslice_e}, 
$$h_{\al\al'}(r,s)r\neq0 \quad\forall~
(r,s)\!\in\!\vph_{\al'}\big(U_{\al}\!\cap\!U_{\al'}\big),~
r\!\in\!\bF^{\fc}\!-\!\{0\},~s\!\in\!\R^m.$$
Since $h_{\al\al'}(0,s)$ is the restriction of 
$\nd_{(0,s)}(\vph_{\al\al';1},\ldots,\vph_{\al\al';\fc})$ to 
$$\R^{\fc}\!\!\times\!\{0\}\subset T_{(0,s)}\big(\vph_{\al'}(U_{\al}\!\cap\!U_{\al'})\!\big),$$
$h_{\al\al'}$ is an isomorphism along the right-hand side in~\eref{Rslice_e}.
Thus, the~map
\BE{wtvphalalprdfn_e}\wt\vph_{\al\al'}\!:
\BLF_Y\big(\vph_{\al'}|_{U_{\al}\cap U_{\al'}}\big)
\lra \BLF_Y\big(\vph_{\al}|_{U_{\al}\cap U_{\al'}}\big), \quad
\wt\vph_{\al\al'}(\ell,x)=\big(h_{\al\al'}(\vph_{\al'}(x)\!)\ell,x\big),\EE
is well-defined and smooth.
It satisfies $\pi_{U_{\al'}}|_{U_{\al}\cap U_{\al'}}\!=\!\pi_{U_{\al}}\!\circ\!\wt\vph_{\al\al'}$.
By the uniqueness of continuous extensions and the cocycle condition
for the overlap maps~$\vph_{\al\al'}$,
\BE{cocyclecond_e}
\wt\vph_{\al\al''}\!=\!\wt\vph_{\al\al'}\!\circ\!\wt\vph_{\al'\al''}\!:
\BLF_Y\big(\vph_{\al''}|_{U_{\al}\cap U_{\al'}\cap U_{\al''}}\big)
\lra \BLF_Y\big(\vph_{\al}|_{U_{\al}\cap U_{\al'}\cap U_{\al''}}\big)
~~\forall\,\al,\al',\al''\!\in\!\cI.\EE

\vspace{.15in}

\noindent
Suppose $\{\vph_{\al}\}_{\al\in\cI}$ is an $\bF$-atlas for~$Y$ in~$X$.
An  \sf{$\bF$-blowup of~$X$ along~$Y$} is the smooth manifold
\begin{gather*}
\BLF_YX\equiv\Big(\!\!(X\!-\!Y)\!\sqcup\!\bigsqcup_{\al\in\cI}\!\!\BLF_Y\vph_{\al}\!
\Big)\!\!\Big/\!\!\!\sim,\quad
\BLF_Y\vph_{\al}\ni(\ell,x)\sim x\in X\!-\!Y~\forall\,x\!\in\!U_{\al}\!-\!Y,\,\al\!\in\!\cI,\\
\BLF_Y\big(\vph_{\al'}|_{U_{\al}\cap U_{\al'}}\big)\ni(\ell,x)\sim
\wt\vph_{\al\al'}(\ell,x)\in
\BLF_Y\big(\vph_{\al}|_{U_{\al}\cap U_{\al'}}\big)
~~\forall~\al,\al'\!\in\!\cI.
\end{gather*}
The \sf{blowdown map}
\BE{blowmapdfn_e}\pi\!:\BLF_YX\lra X, \qquad
\pi\big([\wt{x}]\big)=\begin{cases} 
\wt{x},&\hbox{if}~\wt{x}\!\in\!X\!-\!Y;\\
\pi_{U_{\al}}(\wt{x}),&\hbox{if}~\wt{x}\!\in\!\BLF_Y\vph_{\al},\,\al\!\in\!\cI;
\end{cases}\EE
is well-defined, surjective, proper, and smooth.
It restricts to a diffeomorphism~\eref{standdiff_e}.
The vector bundle isomorphisms
\BE{NXYisom_e} \Phi_{\al}\!:\bF^{\fc}\!\times\!(U_{\al}\!\cap\!Y)\lra \cN_YX|_{U_{\al}\cap Y}, \quad
\Phi_{\al}\big(\!(r_1,\ldots,r_{\fc}),x\big)\lra\sum_{j=1}^{\fc}r_j
 \frac{\prt}{\prt\vph_{\al;j}}\Bigg|_x\!+\!T_xY,\EE
with $\al\!\in\!\cI$ determine a complex structure on  the normal bundle~$\cN_XY$ of~$Y$ in~$X$
in the $\bF\!=\!\C$ case.\\

\noindent
We call two $\bF$-atlases for~$Y$ in~$X$ \sf{$\bF$-equivalent} if their union is still 
a $\bF$-atlas for~$Y$ in~$X$.
Two $\bF$-atlases for~$Y$ in~$X$ that are $\bF$-equivalent determine isomorphic $\bF$-blowups 
of~$X$ along~$Y$.
If $\bF\!=\!\R$, smooth maps~$h_{\al\al'}$ as in~\eref{halalprdfn_e}
exist for any collection $\{\vph_{\al}\}_{\al\in\cI}$ of charts for~$Y$ in~$X$
and for all $\al,\al'\!\in\!\cI$ by~\eref{Rslice_e}.
It follows that any two $\R$-atlases for~$Y$ in~$X$ are $\R$-equivalent.
Thus,
the $\R$-blowup of~$X$ along~$Y$ as constructed just above does not depend on the choice of 
$\R$-atlas for~$Y$ in~$X$.

\subsection{Equivalence}
\label{equivRCbl_subs}

\noindent
We first cut up a tubular neighborhood identification $(W_Y,\Psi_Y)$ for $Y\!\subset\!X$ 
as in Section~\ref{glRCbl_subs} into a compatible collection of coordinate charts
as in Section~\ref{locRCbl_subs} determining a blowup of~$X$ along~$Y$ via 
the construction in the latter section which is isomorphic to 
the blowup determined by $(W_Y,\Psi_Y)$ via the construction in the former section.
The following statement will be used to conclude that the blowups are isomorphic.

\begin{lmm}[{\cite[Lemma~3.1]{RDMbl}}]\label{StanBl_lmm}
Suppose $X$ is a smooth manifold, 
$Y\!\subset\!X$ is a closed submanifold of $\bF$-codimension~$\fc$,
\hbox{$\pi\!:\wt{X}\!\lra\!X$} is a smooth proper map so~that  
\BE{StanBl_e}
\pi\!:\wt{X}\!-\!\pi^{-1}(Y)\lra X\!-\!Y\EE
is a diffeomorphism, and
$\{\vph_{\al}\!\equiv\!(\vph_{\al;1},\ldots,\vph_{\al;\fc+m})\!:
U_{\al}\!\lra\!\bF^{\fc}\!\times\!\R^m\}_{\al\in\cI}$ is an $\bF$-atlas of charts for~$Y$ in~$X$.
Let \hbox{$\pi'\!:\BLF_YX\!\lra\!X$} be the $\bF$-blowup of~$X$ along~$Y$
determined by this atlas as in Section~\ref{locRCbl_subs}.
If there exists a collection 
$$\big\{\wt\phi_{\al;i}\!\equiv\!
\big(\wt\phi_{\al;i;1},\ldots,\wt\phi_{\al;i;\fc+m}\big)
\!:\wt{U}_{\al;i}\!\lra\!\bF^{\fc}\!\times\!\R^m
\big\}_{\al\in\cI,i\in[\fc]}$$
of charts on~$\wt{X}$ so that 
$$\pi^{-1}(U_{\al})=\bigcup_{i=1}^{\fc}\!\wt{U}_{\al;i}\,,
~~
\vph_{\al;j}\!\circ\!\pi|_{\wt{U}_{\al;i}}=\begin{cases}
\wt\phi_{\al;i;j}\!\cdot\!\{\vph_{\al;i}\!\circ\!\pi|_{\wt{U}_{\al;i}}\},
&\hbox{if}~i\!\in\![\fc],\,j\!\in\![\fc]\!-\!\{i\};\\
\wt\phi_{\al;i;j},&\hbox{if}~i\!\in\![\fc],\,
j\!\in\!\{i\}\!\cup\!\big([\fc\!+\!m]\!-\![\fc]);
\end{cases}$$
for every~$\al\!\in\!\cI$, then the maps $\pi$ and $\pi'$ are diffeomorphic.
\end{lmm}

\noindent
Suppose $Y\!\subset\!X$ and $\fc,m\!\in\!\Z^{\ge0}$ are as in Section~\ref{locRCbl_subs}.
If $\cN_XY$ is a complex vector bundle, let $\bF\!=\!\C$;
otherwise, let $\bF\!=\!\R$.
Let $(W_Y,\Psi_Y)$ be a tubular neighborhood identification for $Y\!\subset\!X$ 
over an open subset $U_Y\!\subset\!X$. 
Choose a collection
\BE{Phialdfn_e}\Phi_{\al}\!\equiv\!\big(\Phi_{\al;1},\ldots,\Phi_{\al;\fc+m}\big)\!:
\cN_XY|_{U_{Y;\al}}\lra \bF^{\fc}\!\times\!\R^m, \qquad \al\!\in\!\cI,\EE
of $\bF$-trivializations of~$\cN_XY$ so that the open subsets $U_{Y;\al}\!\subset\!Y$
cover~$U_Y$.
In particular, 
$$U_{Y;\al}=\big\{v\!\in\!\cN_XY|_{U_{Y;\al}}\!:\Phi_{\al;1}(v),\ldots,\Phi_{\al;\fc}(v)\!=\!0\big\},
~~
\Phi_{\al;\fc+i}\!=\!\Phi_{\al;\fc+i}\!\circ\!\pi_{\cN_XY}\!:\cN_XY\lra\R
~\forall\,i\!\in\![m],$$
and the restrictions
$$\big(\Phi_{\al;1},\ldots,\Phi_{\al;\fc}\big)\!:\cN_XY|_y\lra\bF^{\fc},
\quad y\!\in\!U_{Y;\al},$$
are $\bF$-linear isomorphisms.
Thus, for all $\al,\al'\!\in\cI$ there exists a smooth map 
\begin{gather}\label{halalprequiv_e}
h_{\al\al'}\!:U_{Y;\al}\!\cap\!U_{Y;\al'}\lra\Aut_{\bF}(\bF^{\fc})\qquad\hbox{s.t.}\\
\notag
\Phi_{\al}\!\circ\!\Phi_{\al'}^{-1}\big(r,\{\!(\Phi_{\al';\fc+i})_{i\in[m]}\}(x)\!\big)
=\big(\!\big\{h_{\al\al'}(x)\!\big\}(r),\{\!(\Phi_{\al;\fc+i})_{i\in[m]}\}(x)\!\big)
~\forall\,(r,x)\!\in\!\bF^c\!\times\!\big(U_{Y;\al}\!\cap\!U_{Y;\al'}\big).
\end{gather}
For each $\al\!\in\!\cI$, let
\BE{equivchartdfn_e}U_{\al}\!\equiv\!\Psi_Y\big(W_Y|_{U_{Y;\al}}\big), \quad
\vph_{\al}\!\equiv\!\big(\vph_{\al;1},\ldots,\vph_{\al;\fc+m}\big)
\!\equiv\!\Phi_{\al}\!\circ\!\Psi_Y^{-1}\big|_{U_{\al}}\!:U_{\al}\lra\bF^{\fc}\!\times\!\R^m\,.\EE
This is a coordinate chart for~$Y$ in~$X$; we call it a \sf{chart cut out from~$(W_Y,\Psi_Y)$}.
The collection $\{\vph_{\al}\!:U_{\al}\!\lra\!\bF^{\fc}\!\times\!\R^m\}_{\al\in\cI}$
of these charts covers~$U_Y$.
If $U_Y\!=\!Y$, then $\{\vph_{\al}\}_{\al\in\cI}$ is an $\bF$-atlas for~$Y$ in~$X$.
The associated maps~$h_{\al\al'}$ in~\eref{halalprdfn_e} in this case do not depend on $r\!\in\!\bF^{\fc}$.
It is immediate that $\bF$-equivalent tubular neighborhood identifications for $Y\!\subset\!X$ 
produce $\bF$-atlases for $Y\!\subset\!X$ that are $\bF$-equivalent.\\

\noindent
Suppose now that $U_Y\!=\!Y$ and \hbox{$\pi\!:\wt{X}\!\lra\!X$}
is the corresponding blowup of~$X$ along~$Y$ as in Section~\ref{glRCbl_subs}.
The composition of the projection in~\eref{gacNbFdfn_e} with $\cN\!=\!\cN_XY$ 
and the bundle projection \hbox{$\bF\P(\cN_XY)\!\lra\!Y$} induces a projection
\hbox{$\ga_{\cN_XY}^{\bF}\!\lra\!Y$}.
For $\al\!\in\!\cI$, let
$$\wt{W}_Y^{\bF}\big|_{U_{Y;\al}}\subset \wt{W}_Y^{\bF}\subset\ga_{\cN_XY}^{\bF}$$
be the preimage of $U_{Y;\al}\!\subset\!Y$ under the restriction of this composition 
to~$\wt{W}_Y^{\bF}$;
this is also the preimage of $U_{\al}\!\subset\!X$ under the blowdown map 
\hbox{$\pi\!:\wt{X}\!\lra\!X$}.
For each $i\!\in\![\fc]$, the subspace 
$$\wt{U}_{\al;i}\equiv\big\{\!\big([r],v\big)\!\in\!\wt{W}_Y^{\bF}|_{U_{Y;\al}}\!:
\Phi_{\al;i}(r)\!\neq\!0\big\}\subset \wt{W}_Y^{\bF}\subset \wt{X}$$
is open.
The~map
\begin{gather*}
\wt\phi_{\al;i}\!\equiv\!\big(\wt\phi_{\al;i;1},\ldots,\wt\phi_{\al;i;\fc+m}\big)\!:
\wt{U}_{\al;i}\lra \bF^{\fc}\!\times\!\R^m,\\
\wt\phi_{\al;i;j}\big([r],v\big)
=\begin{cases}\Phi_{\al;j}(r)/\Phi_{\al;i}(r),&\hbox{if}~j\!\in\![\fc]\!-\!\{i\};\\
\Phi_{\al;j}(v),&\hbox{if}~j\!\in\!\{i\}\!\cup\!\big([\fc\!+\!m]\!-\![\fc]);
\end{cases}
\end{gather*}
is a coordinate chart on~$\wt{X}$.
The collection of these charts satisfies the conditions of Lemma~\ref{StanBl_lmm}.
Thus, the $\bF$-blowup~$\wt{X}$ obtained from the tubular neighborhood identification~$(W_Y,\Psi_Y)$
for $Y\!\subset\!X$ as in Section~\ref{glRCbl_subs} 
is isomorphic to the $\bF$-blowup~$\BLF_YX$ obtained from the above $\bF$-atlas
$\{\vph_{\al}\}_{\al\in\cI}$ cut up from~$(W_Y,\Psi_Y)$.\\

\noindent
By Proposition~\ref{Loc2GlRC_prp} below, an $\bF$-atlas $\{\vph_{\al}\}_{\al\in\cI}$ of charts 
for~$Y$ in~$X$ can be assembled into a tubular neighborhood identification~$(W_Y,\Psi_Y)$ for $Y\!\subset\!X$
so that any $\bF$-atlas cut out from~$(W_Y,\Psi_Y)$ is $\bF$-equivalent
to~$\{\vph_{\al}\}_{\al\in\cI}$ and thus the blowups obtained via the local construction
of Section~\ref{locRCbl_subs} from the two $\bF$-atlases and 
the blowup obtained via the global construction of Section~\ref{glRCbl_subs} are isomorphic.
Furthermore, tubular neighborhood identifications assembled from  $\bF$-equivalent $\bF$-atlases
are $\bF$-equivalent.
This establishes a correspondence between the global and local blowup constructions
of Sections~\ref{glRCbl_subs} and~\ref{locRCbl_subs}, respectively.

\begin{prp}\label{Loc2GlRC_prp}
Suppose $X$ is a smooth manifold, $Y\!\subset\!X$ is a closed submanifold, and
$\{\vph_{\al}\}_{\al\in\cI}$ is an $\bF$-atlas of charts for~$Y$ in~$X$.
There exists a tubular neighborhood identification~$(W_Y,\Psi_Y)$ for~$Y\!\subset\!X$
such that any $\bF$-atlas of charts cut out from~$(W_Y,\Psi_Y)$ is $\bF$-equivalent
to~$\{\vph_{\al}\}_{\al\in\cI}$.
\end{prp}

\begin{proof} We will call a chart $\vph\!:U\!\lra\!\bF^{\fc}\!\times\!\R^m$ for~$Y$ in~$X$
\sf{$\bF$-compatible} with the $\bF$-atlas $\{\vph_{\al}\}_{\al\in\cI}$
if the union of this atlas with $\{\vph\}$ is still an $\bF$-atlas for~$Y$ in~$X$.
We will call a tubular neighborhood identification~$(W_Y,\Psi_Y)$ for $Y\!\subset\!X$
over an open subset $U_Y\!\subset\!Y$ \sf{$\bF$-compatible} with $\{\vph_{\al}\}_{\al\in\cI}$
if any coordinate chart for $Y\!\subset\!X$ cut out from~$(W_Y,\Psi_Y)$ is 
$\bF$-compatible with $\{\vph_{\al}\}_{\al\in\cI}$.\\

\noindent
For each $\al\!\in\!\cI$, let
$$\vph_{\al}\!=\!(\vph_{\al;1},\ldots,\vph_{\al;\fc+m})\!:
U_{\al}\!\lra\!\bF^{\fc}\!\times\!\R^m$$
as before, $U_{Y;\al}\!=\!U_{\al}\!\cap\!Y$, and $\Phi_{\al}$ be as in~\eref{NXYisom_e}.
We define a tubular neighborhood identification for $Y\!\subset\!X$ over~$U_{Y;\al}$ by
\begin{gather*}
\Psi_{\al}\!:W_{\al}\!\equiv\!\Phi_{\al}\big(\!\big\{\!(r,y)\!\in\!\bF^c\!\times\!U_{Y;\al}\!:
\big(r,\{\!(\vph_{\al;\fc+1},\ldots,\vph_{\al;\fc+m})\!\}(y)\!\big)\!\in\!\vph_{\al}(U_{\al})\!\big\}
\!\big)
\lra X,\\
\Psi_{\al}\big(\Phi_{\al}(r,y)\!\big)=
\vph_{\al}^{-1}\big(r,\{\!(\vph_{\al;\fc+1},\ldots,\vph_{\al;\fc+m})\!\}(y)\!\big).
\end{gather*}
This identification is $\bF$-compatible with the $\bF$-atlas $\{\vph_{\al'}\}_{\al'\in\cI}$
because the domain of any chart cut out from~$(W_{\al},\Psi_{\al})$ is contained in~$U_{\al}$
and any such chart is compatible with~$\vph_{\al}$ in the sense of~\eref{halalprdfn_e}.
The collection $\{Y\!\cap\!W_{\al}\}_{\al\in\cI}$ is an open cover of~$Y$.\\

\noindent
Let $\{U_{Y;i}\}_{i\in[N]}$ be an open cover of~$Y$ and
$\{\!(W_i,\Psi_i)\!\}_{i\in[N]}$ be a collection of tubular neighborhood identifications
for $Y\!\subset\!X$ over $U_{Y;i}\!\subset\!Y$ so that 
each $(W_i,\Psi_i)$ is $\bF$-compatible with the $\bF$-atlas~$\{\vph_{\al'}\}_{\al'\in\cI}$.
This in particular implies that for any $i,j\!\in\![N]$ 
the tubular neighborhood identifications $(W_i,\Psi_i)$ and~$(W_j,\Psi_j)$
are $\bF$-equivalent.
Such collections $\{U_{Y;i}\}_{i\in[N]}$ and $\{\!(W_i,\Psi_i)\!\}_{i\in[N]}$ with $N\!=\!m\!+\!1$
can be obtained by taking each~$U_{Y;i}$ to be a disjoint union of open subsets~$U_{Y;\be}$ of~$Y$
each of which is contained in some~$U_{Y;\al(\be)}$ and then restricting~$\Psi_{\al(\be)}$
to a neighborhood of~$U_{Y;\be}$ in~$W_{\al(\be)}$;
see the proof of \cite[Lemma~2.4]{pseudo} for a similar construction.
The claim now follows from Lemma~\ref{Loc2GlRCmerge_lmm} below applied $N\!-\!1$ times.
\end{proof}

\begin{lmm}\label{Loc2GlRCmerge_lmm}
Suppose $X$ is a smooth manifold, $Y\!\subset\!X$ is a closed submanifold,
$\cN_XY$ has a complex structure if $\bF\!=\!\C$,
and $(W_1,\Psi_1)$ and $(W_2,\Psi_2)$ are tubular neighborhood identifications for $Y\!\subset\!X$
over open subsets $U_1\!\subset\!Y$ and $U_2\!\subset\!Y$, respectively.
If $(W_1,\Psi_1)$ and $(W_2,\Psi_2)$ are $\bF$-equivalent,
there exists a tubular neighborhood identification~$(W,\Psi)$ for~$Y$ over 
\hbox{$U_1\!\cup\!U_2$} which is $\bF$-equivalent to~$(W_1,\Psi_1)$ and~$(W_2,\Psi_2)$.
\end{lmm}

\begin{proof} The proof is an adaptation of the proof of \cite[Lemma~4.4]{SympDivConf}.
Let \hbox{$\pi_{TY}\!:TY\!\lra\!Y$} be the bundle projection and 
$T^{\ver}(TY)\!\subset\!T(TY)$ be the vertical tangent subbundle;
the fiber~$T_v^{\ver}(TY)$ of the latter at $v\!\in\!TY$ is canonically identified with $T_{\pi_{TY}(v)}Y$.
Let \hbox{$\exp_Y\!:TY\!\lra\!Y$} be a smooth map so that 
$\exp_Y|_{T_yY}$ is an open embedding for every $y\!\in\!Y$, 
\BE{expcond_e}\exp_Y\!\big|_Y=\id_Y, \qquad\hbox{and}\qquad
\nd_y\exp_Y\!=\!\id\!: T_y^{\ver}(TY)\lra T_yY~~\forall\,y\!\in\!Y. \EE
Choose an isomorphism
\BE{wtexpcond_e0}\wt\exp_Y\!:\pi_{TY}^*\cN_XY\lra \exp_Y^*\cN_XY\subset TY\!\times\!\cN_XY\EE
of $\bF$-vector bundles over~$TY$ restricting to the identity over~$Y\!\subset\!TY$, i.e.
\BE{wtexpcond_e}\wt\exp_Y(y,v)=(y,v)\qquad\forall~
(y,v)\in\big(\pi_{TY}^*\cN_XY\big)\!\big|_Y\subset  Y\!\times\!\cN_XY.\EE
Denote by $\pi_2\!:\exp_Y^*\cN_XY\!\lra\!\cN_XY$ 
the projection to the second component.\\

\noindent
Let $W_{\cap}\!\subset\!\cN_XY|_{U_1\cap U_2}$ and~$h$ be as in~\eref{TNIequiv_e} and
$$W_{\cap}'=\big\{v\!\in\!W_{\cap}\!:
\pi_{\cN_XY}\big(\Psi_1^{-1}(\Psi_2(v)\!)\!\big)\!\in\!\exp_Y(T_{\pi_{\cN_XY}(v)}Y)\big\};$$
the last set is an open neighborhood of $U_1\!\cap\!U_2$ in~$\cN_XY$.
Define 
\begin{alignat*}{2}
\th&\in\Ga\big(W_{\cap}',\pi_{\cN_XY}^*TY\big) &\qquad\hbox{by}\quad
&\exp_Y\!\big(\th(v)\!\big)=\pi_{\cN_XY}\big(\Psi_1^{-1}(\Psi_2(v)\!)\!\big),\\
h_{\th}&\in\Ga\big(W_{\cap}',\pi_{\cN_XY}^*\End_{\bF}(\cN_XY)\!\big)
&\qquad\hbox{by}\quad
&\pi_2\big(\wt\exp_Y\!\big(\th(v),\big\{\!h_{\th}(v)\!\big\}(w)\!\big)\!\big)
=\big\{\!h(v)\!\big\}(w).
\end{alignat*}
By~\eref{TNIprop_e},
\BE{TNIprop_e2}\begin{split}
&\th\big|_{U_1\cap U_2}\!=\!\id_{U_1\cap U_2}\!:
U_1\!\cap\!U_2\lra U_1\!\cap\!U_2\subset \pi_{\cN_XY}^*TY\big|_Y \qquad\hbox{and}\\
&h_{\th}\big|_{U_1\cap U_2}\!=\!\id_{\cN_XY|_{U_1\cap U_2}}\!:
\cN_XY|_{U_1\cap U_2}\lra\cN_XY|_{U_1\cap U_2}\,.
\end{split}\EE

\vspace{.18in} 

\noindent
Replacing the manifold $X$ with an open subset~$X'$ which intersects~$Y$ at $U_1\!\cap\!U_2$ and 
each tubular neighborhood identification $(W_i,\Psi_i)$ by $(\Psi_i^{-1}(X'),\Psi_i|_{\Psi_i^{-1}(X')})$, 
we can assume that \hbox{$U_1\!\cup\!U_2\!=\!Y$}.
Choose open subsets $U_1'\!\subset\!U_1$ and $U_2'\!\subset\!U_2$ so~that
$$\ov{U_1'}\subset U_1', \qquad \ov{U_2'}\subset U_2', 
\quad\hbox{and}\quad U_1'\!\cup\!U_2'=Y$$
and a smooth bump function $\eta\!:Y\!\lra\![0,1]$ with support contained in~$U_2'$
so that \hbox{$\eta|_{U_2-\ov{U_1'}}\!=\!1$}.
The subset 
$$W\equiv W_1\big|_{U_1-\ov{U_2'}}\!\cup\!W_{\cap}'\!\cup\!W_2\big|_{U_2-\ov{U_1'}}
\subset \cN_XY$$
is then an open neighborhood of~$Y$ and
\begin{gather*}
\Psi\!:W\lra \cN_XY, \\
\Psi(v)=\begin{cases}
\Psi_1(v),&\hbox{if}~v\!\in\!W_1|_{U_1-\ov{U_2'}};\\
\Psi_1
\big(\pi_2\big(\wt\exp_Y\big(\eta(\pi_{\cN_XY}(v)\!)\!\big)\th(v),
v\!+\!\eta(\pi_{\cN_XY}(v)\!)\big(\{h_{\th}(v)\!\}(v)\!-\!v\big)\!\big)\!\big),
&\hbox{if}~v\!\in\!W_{\cap}'\,;\\
\Psi_2(v),&\hbox{if}~v\!\in\!W_2|_{U_2-\ov{U_1'}};
\end{cases}
\end{gather*}
is a well-defined smooth map.
By~\eref{TNIprop_e2}, $\Psi|_Y\!=\!\id_Y$ and the homomorphism
$$T(\cN_XY)\big|_Y \xlra{\nd\Psi} TX|_Y
\xlra{\pi_Y^{\perp}} \frac{TX|_Y}{TY}\!\equiv\!\cN_XY$$
restricts to the identity homomorphism on \hbox{$\cN_XY\!\subset\!T(\cN_XY)|_Y$}.
By the Inverse Function Theorem, there thus exists a neighborhood $W'\!\subset\!W$
of~$Y$ such that $(W',\Psi|_{W'})$ is a tubular neighborhood identification for $Y\!\subset\!X$.\\

\noindent
Since $U_1\!\subset\!(U_1-\ov{U_2'})\!\cup\!W_{\cap}'$ and 
$$\Psi_1^{-1}\big(\Psi(v)\!\big)=\begin{cases}
v,&\hbox{if}~v\!\in\!W'|_{U_1-\ov{U_2'}};\\
\pi_2\big(\wt\exp_Y\big(\eta(\pi_{\cN_XY}(v)\!)\!\big)\th(v),
v\!+\!\eta(\pi_{\cN_XY}(v)\!)\big(\{h_{\th}(v)\!\}(v)\!-\!v\big)\!\big),
&\hbox{if}~v\!\in\!W'\!\cap\!W_{\cap}';
\end{cases}$$
the tubular neighborhood identification $(W',\Psi|_{W'})$ is $\bF$-equivalent to~$(W_1,\Psi_1)$. 
The restrictions of~$\Psi'|_{W'}$ and~$\Psi_2$ to~$W'|_{U_2-\ov{U_1'}}$ are the same.
Since $Y\!\cap\!W_{\cap}'\!\subset\!U_1$,
the restriction of~$\Psi'|_{W'}$ to~$W'\!\cap\!W_{\cap}'$ is $\bF$-equivalent to~$(W_1,\Psi_1)$,
and~$(W_1,\Psi_1)$ is $\bF$-equivalent to~$(W_2,\Psi_2)$,
the restriction of~$\Psi'|_{W'}$ to~$W'\!\cap\!W_{\cap}'$ is $\bF$-equivalent to~$(W_2,\Psi_2)$.
Since $U_2\!\subset\!(U_2-\ov{U_1'})\!\cup\!W_{\cap}'$, it follows that
$(W',\Psi|_{W'})$ is $\bF$-equivalent to~$(W_2,\Psi_2)$ as~well.
\end{proof}

\section{Augmented blowup}
\label{glaugbl_sec}

\noindent
An $\cN^{\fc_1}$-augmented blowup is a pair of blowups followed by a blowdown.
This makes the global description of this blowup significantly more extensive 
than in Section~\ref{glRCbl_subs} and requires a preparatory step
ensuring that the blowdown step can be carried~out in the smooth category.
The local description is even more cumbersome as it involves two new types
of coordinate charts, in addition to a suitable $\R$-atlas of charts on the original manifold,
and overlap maps between~these charts.

\subsection{Global construction}
\label{glaugbl_subs}

\noindent
We continue with the notation of the global blowup construction 
in Section~\ref{glRCbl_subs}, restricting to the case $\bF\!=\!\R$.

\subsubsection{Preparation}
Let $\cN\!\lra\!Y$ be a real vector bundle of rank~$\fc$ and
$\cN^{\fc_1}\!\subset\!\cN$ be a real subbundle of corank~$\fc_1$
for some $\fc_1\!\in\![\fc\!-\!1]$.
Suppose $\lr{\cdot,\cdot}$ is an inner-product on~$\cN^{\fc_1}$ and
\hbox{$\cN'\!\subset\!\cN$} is a subbundle complementary to~$\cN^{\fc_1}$ 
so that \hbox{$\cN\!=\!\cN'\!\oplus\!\cN^{\fc_1}$}.
The projectivization~$\R\P\cN$ then contains~$\R\P\cN'$ and~$\R\P\cN^{\fc_1}$
as disjoint submanifolds.
Let 
$$ \pi'\!:\R\P\cN'\lra Y, \quad \pi_{\fc_1}\!:\R\P\cN^{\fc_1}\!\lra\!Y, \quad\hbox{and}\quad 
\wt\pi',\wt\pi_{\fc_1}\!:\R\P\cN'\!\times_Y\!\R\P\cN^{\fc_1}\lra\R\P\cN',\R\P\cN^{\fc_1}$$
be the projections.\\

\noindent 
There is a vector bundle isomorphism
\BE{RRcNbl_e1}\cN_{\R\P\cN}(\R\P\cN^{\fc_1})\approx 
(\ga_{\cN^{\fc_1}}^{\R})^{\!*}\!\!\otimes_{\R}\!\pi_{\fc_1}^*\cN'
\lra\R\P\cN^{\fc_1}\subset\R\P\cN.\EE
This isomorphism is associated with the tubular neighborhood identification
\BE{RRcNbl_e3}\Psi_{\R\P\cN^{\fc_1}}\!: 
(\ga_{\cN^{\fc_1}}^{\R})^{\!*}\!\!\otimes_{\R}\!\pi_{\fc_1}^*\cN'
\lra \R\P\cN\!-\!\R\P\cN', \quad
\Psi_{\R\P\cN^{\fc_1}}\big(u\!\otimes\!([w],v)\!\big)=\big[u(w)v\!+\!w\big].\EE
We blow up~$\R\P\cN$ along~$\R\P\cN^{\fc_1}$ using this tubular neighborhood identification
as in Section~\ref{glRCbl_subs}.
The isomorphism~\eref{RRcNbl_e1} yields identifications
\begin{equation*}
\xymatrix{\cN_{\BLR_{\R\P\cN^{\fc_1}}\!(\R\P\cN)}\!
\big(\bE_{\R\P\cN^{\fc_1}}^{\R}\!(\R\P\cN)\!\big)
\!\equiv\!\ga_{\cN_{\R\P\cN}(\R\P\cN^{\fc_1})} 
\ar[r]^<<<<<{\approx}\ar[d]&
\wt\pi_{\fc_1}^*(\ga_{\cN^{\fc_1}}^{\R})^*\!\otimes\!\wt\pi'^*\ga_{\cN'}^{\R}
\!=\!\wt\pi'^*\ga_{\cN'}^{\R}\!\otimes\!(\ga_{\pi'^*\cN^{\fc_1}}^{\R})^*
\ar[d]\\
\bE_{\R\P\cN^{\fc_1}}^{\R}\!(\R\P\cN)
\!\equiv\! \R\P\big(\cN_{\R\P\cN}(\R\P\cN^{\fc_1})\!\big)\ar[r]^<<<<<<<<<{\approx}&
\R\P\cN'\!\times_Y\!\R\P\cN^{\fc_1}\!=\!\R\P\big(\ga_{\cN'}^{\R}\!\otimes\!\pi'^*\cN^{\fc_1}\big).}
\end{equation*}
These identifications commute with the projections to~$\R\P\cN^{\fc_1}$.
Since  \hbox{$(\ga_{\pi'^*\cN^{\fc_1}}^{\R})^{\!*}\!\approx\!\ga_{\pi'^*\cN^{\fc_1}}^{\R}$} 
via~$\lr{\cdot,\cdot}$,
these identifications induce an identification
\BE{BRREdiag_e}\begin{split}
\xymatrix{\cN_{\BLR_{\R\P\cN^{\fc_1}}\!(\R\P\cN)}\!\big(\bE_{\R\P\cN^{\fc_1}}^{\R}\!(\R\P\cN)\!\big)
\!-\!\bE_{\R\P\cN^{\fc_1}}^{\R}\!(\R\P\cN)  \ar[d]\ar[rr]^>>>>>>>>>>>{\approx}&&
 \ga_{\cN'}^{\R}\!\otimes\!\pi'^*\cN^{\fc_1}\!-\!\R\P\cN'\ar[d]\\
\bE_{\R\P\cN^{\fc_1}}^{\R}\!(\R\P\cN)\ar[rr]^{\wt\pi'} && \R\P\cN'}
\end{split}\EE
which commutes with the projections to~$Y$.\\

\noindent
By the previous paragraph,
the quotient~$\BL_{\cN^{\fc_1}}\!(\R\P\cN)$ of~$\BLR_{\R\P\cN^{\fc_1}}\!(\R\P\cN)$ obtained by
collapsing each fiber of the map~$\wt\pi'$ in~\eref{BRREdiag_e} to a point is a smooth manifold.
The images~of 
$$\R\P\cN'\subset\R\P\cN\!-\!\R\P\cN^{\fc_1}= 
\BLR_{\R\P\cN^{\fc_1}}\!(\R\P\cN)\!-\!\bE_{\R\P\cN^{\fc_1}}^{\R}\!(\R\P\cN)$$
and $\bE_{\R\P\cN^{\fc_1}}^{\R}\!(\R\P\cN)$ under this collapsing procedure are smooth 
submanifolds of~$\BL_{\cN^{\fc_1}}\!(\R\P\cN)$ canonically identified with~$\R\P\cN'$.
We denote them by~$\R\P\cN'_0$ and~$\R\P\cN'_{\i}$, respectively.
Their normal bundles are  
$(\ga_{\cN'}^{\R})^*\!\otimes_{\R}\!\pi'^*\cN^{\fc_1}$ and
$\ga_{\cN'}^{\R}\!\otimes_{\R}\!\pi'^*\cN^{\fc_1}$, respectively.
The manifold~$\BL_{\cN^{\fc_1}}\!(\R\P\cN)$ is obtained by gluing the complements 
of the zero sections in the total spaces of these bundles:
\begin{gather*}
\BL_{\cN^{\fc_1}}\!(\R\P\cN)=
\big(\big(\!(\ga_{\cN'}^{\R})^*\!\otimes_{\R}\!\pi'^*\cN^{\fc_1})\!\sqcup\!
\big(\ga_{\cN'}^{\R}\!\otimes_{\R}\!\pi'^*\cN^{\fc_1}\big)\!\big)
\!\big/\!\!\sim,\\
(\ga_{\cN'}^{\R})^*\!\otimes_{\R}\!\pi'^*\cN^{\fc_1}\!-\!\R\P\cN'_0
\ni u\!\otimes\!\big([v],w\big) \sim 
\big([v],v/u(v)\!\big)\!\otimes\!\big([v],w/\lr{w,w}\!\big)
\in \ga_{\cN'}^{\R}\!\otimes_{\R}\!\pi'^*\cN^{\fc_1}\!-\!\R\P\cN'_{\i}.
\end{gather*}
Thus, there is a natural fiber bundle
$$S^{\fc-\fc_1}\lra \BL_{\cN^{\fc_1}}\!(\R\P\cN)\lra\R\P\cN'\approx \R\P\big(\cN_XY/\cN^{\fc_1}\big)$$
with disjoint sections~$\R\P\cN'_0$ and~$\R\P\cN'_{\i}$.
Let
$$\R\P_{\cN^{\fc_1}}\cN=\BL_{\cN^{\fc_1}}\!(\R\P\cN)
\underset{\R\P\cN'_{\i}=\R\P(\cN'\oplus\tau_Y^0)}{\cup}\R\P\big(\cN'\!\oplus\!\tau_Y^1\big).$$

\subsubsection{Construction}\label{glaugbl_subsubs}
Let $\cN^{\fc_1}$ be a real subbundle of $\cN\!\equiv\!\cN_XY$ 
of corank~$\fc_1$ for some \hbox{$\fc_1\!\in\![\fc\!-\!1]$}.
Suppose $\lr{\cdot,\cdot}$ is an inner-product on~$\cN^{\fc_1}$ and
$\cN'\!\subset\!\cN_XY$ is a complement of~$\cN^{\fc_1}$ as before.
Let 
\begin{gather*}
\pi'\!:\R\P\big(\cN'\!\oplus\!\tau_Y^1\big)\lra Y, \qquad
\pi_{\fc_1}\!:\R\P\cN^{\fc_1}\lra Y, \qquad\hbox{and}\\ 
\wt\pi',\wt\pi_{\fc_1}\!:\R\P\big(\cN'\!\oplus\!\tau_Y^1\big)\!\times_Y\!\R\P\cN^{\fc_1}
\lra \R\P\big(\cN'\!\oplus\!\tau_Y^1\big),\R\P\cN^{\fc_1}
\end{gather*}
be the projections.\\

\noindent
We first blow up~$X$ along~$Y$ using a tubular neighborhood identification~$(W_Y,\Psi_Y)$
as in Section~\ref{glRCbl_subs}.
The isomorphism~\eref{RRcNbl_e1} then extends to a vector bundle isomorphism
$$\cN_{\BLR_YX}\big(\R\P\cN^{\fc_1}\big) \approx 
\big(\ga_{\cN^{\fc_1}}^{\R}\big)^{\!*}\!\otimes_{\R}\!\pi_{\fc_1}^*\cN'
\!\oplus\!\ga_{\cN^{\fc_1}}^{\R}\lra \R\P\cN^{\fc_1}\subset\BLR_YX.$$
This identification is associated with the restriction of
the tubular neighborhood identification
\begin{gather*}
\Psi_{\fc_1}\!: \big(\ga_{\cN^{\fc_1}}^{\R}\big)^{\!*}\!\otimes_{\R}\!\pi_{\fc_1}^*\cN'
\!\oplus\!\ga_{\cN^{\fc_1}}^{\R}\lra \ga_{\cN}^{\R}, \\
\Psi_{\fc_1}\big(u\!\otimes\!([w],v),([w],w')\!\big)
=\big([w\!+\!u(w)v],u(w')v\!+\!w'\big),
\end{gather*}
to the preimage of $\wt{W}_Y^{\R}\!\subset\!\ga_{\cN}^{\R},\BLR_YX$.
The inner-product $\lr{\cdot,\cdot}$ thus induces an isomorphism
\BE{RRcNbl_e1b}\cN_{\BLR_YX}\big(\R\P\cN^{\fc_1}\big)\approx
\ga_{\cN^{\fc_1}}^{\R}\!\otimes\! \pi_{\fc_1}^*\big(\cN'\!\oplus\!\tau_Y^1\big)\EE
of vector bundles over~$\R\P\cN^{\fc_1}$.\\

\noindent
We next blow up~$\BLR_YX$ along~$\R\P\cN^{\fc_1}$ using the above tubular neighborhood 
identification as in Section~\ref{glRCbl_subs}.
The isomorphism~\eref{RRcNbl_e1b} yields identifications
\begin{small}\begin{equation*}
\xymatrix{\cN_{\BLR_{\R\P\cN^{\fc_1}}\!(\BLR_YX)}\!\big(\bE_{\R\P\cN^{\fc_1}}^{\R}\!(\BLR_YX)\!\big)
\!\equiv\!\ga_{\cN_{\BLR_YX}(\R\P\cN^{\fc_1})} \ar[r]^<<<<{\approx}\ar[d]&
\wt\pi_{\fc_1}^*\ga_{\cN^{\fc_1}}^{\R}\!\otimes\!\wt\pi'^*\ga_{\cN'\oplus\tau_Y^1}
\!=\!\wt\pi'^*\ga_{\cN'\oplus\tau_Y^1}^{\R}\!\otimes\!\ga_{\pi'^*\cN^{\fc_1}}^{\R}\ar[d]\\
\bE_{\R\P\cN^{\fc_1}}^{\R}\!(\BLR_YX)
\!\equiv\! \R\P\big(\cN_{\BLR_YX}(\R\P\cN^{\fc_1})\!\big)\ar[r]^<<<<<<<<<{\approx}&
 \R\P\big(\cN'\!\oplus\!\tau_Y^1\big)\!\times_Y\!\R\P\cN^{\fc_1}
\!=\!\R\P\big(\ga_{\cN'\oplus\tau_Y^1}^{\R}\!\otimes\!\pi'^*\cN^{\fc_1}\big).}
\end{equation*}
\end{small}The above identifications commute with the projections to~$\R\P\cN^{\fc_1}$
and induce an identification 
$$\xymatrix{\cN_{\BLR_{\R\P\cN^{\fc_1}}\!(\BLR_YX)}\!\big(\bE_{\R\P\cN^{\fc_1}}^{\R}\!(\BLR_YX)\!\big)
\!-\!\bE_{\R\P\cN^{\fc_1}}^{\R}\!(\BLR_YX)  \ar[d]\ar[rr]^>>>>>>>>>>>{\approx}&&
\ga_{\cN'\oplus\tau_Y^1}^{\R}\!\otimes\!\pi'^*\cN^{\fc_1}
\!-\!\R\P\big(\cN'\!\oplus\!\tau_Y^1\big)\ar[d]\\
\bE_{\R\P\cN^{\fc_1}}^{\R}\!(\BLR_YX)\ar[rr]^{\wt\pi'} &&\R\P\big(\cN'\!\oplus\!\tau_Y^1\big)}$$
which commutes with the projections to~$Y$.\\

\noindent
By the previous paragraph,
the quotient~$\BL_{\cN^{\fc_1}}\!X$ of~$\BLR_{\R\P\cN^{\fc_1}}\!(\BLR_YX)$ obtained by
collapsing each fiber of~$\wt\pi'$ to a point is a smooth manifold;
we call it the \sf{$\cN^{\fc_1}$-augmented blowup of~$X$ along~$Y$}.
This manifold contains 
$$\bE_{\cN^{\fc_1}}^-X\equiv\BL_{\cN^{\fc_1}}\!\big(\R\P(\cN_XY)\!\big)
\qquad\hbox{and}\qquad 
\bE_{\cN^{\fc_1}}^0X\equiv\R\P\big(\cN'\!\oplus\!\tau_Y^1\big)$$ 
as smooth submanifolds intersecting along $\R\P\cN'_{\i}\!=\!\R\P(\cN'\!\oplus\!\tau_Y^0)$.
Thus, the \sf{exceptional locus} 
$$\bE_{\cN^{\fc_1}}\!X\equiv \bE_{\cN^{\fc_1}}^-X\!\cup\!\bE_{\cN^{\fc_1}}^0X
\subset \BL_{\cN^{\fc_1}}\!X$$
is isomorphic to $\R\P_{\cN^{\fc_1}}\!(\cN_XY)$ (as a union of two smooth manifolds
joined along a common submanifold).
The composition
$$ \BLR_{\R\P\cN^{\fc_1}}\!(\BLR_YX) \stackrel\pi\lra \BLR_YX \stackrel\pi\lra X$$
descends to a \sf{blowdown map} $\pi\!:\BL_{\cN^{\fc_1}}\!X\!\lra\!X$,
which is surjective, proper, and smooth.
The latter restricts to a diffeomorphism from $\BL_{\cN^{\fc_1}}\!X\!-\!\bE_{\cN^{\fc_1}}\!X$
to~$X\!-\!Y$.\\

\noindent
We call tubular neighborhood identifications $(W_1,\Psi_1)$ and~$(W_2,\Psi_2)$ 
for $Y\!\subset\!X$ over open subsets $U_1\!\subset\!Y$ and $U_2\!\subset\!Y$, respectively,
\sf{$(\cN^{\fc_1},\lr{\cdot,\cdot}\!)$-equivalent} if there exist~$W_{\cap}$ and $h$ in~\eref{TNIequiv_e}
with $\bF\!=\!\R$ and $f\!\in\!C^{\i}(W_{\cap};\R^+)$ such~that
\BE{augatlcond_e0}\begin{aligned}
\big\{\!h(v)\!\big\}\!\big(\cN^{\fc_1}|_{\pi_{\cN_XY}(v)}\big)
&\subset\cN^{\fc_1}\big|_{\pi_{\cN_XY}(\Psi_1^{-1}(\Psi_2(v)))}
&\qquad&\forall~v\!\in\!W_{\cap},\\
\blr{\!\{h(v)\!\}(w),\{h(v)\!\}(w)\!}&=f(v)\lr{w,w}
&\qquad&\forall~v\!\in\!W_{\cap},\,w\!\in\!\cN^{\fc_1}\big|_{\pi_{\cN_XY}(v)}.
\end{aligned}\EE
This notion depends only on the conformal equivalence class of the inner-product~$\lr{\cdot,\cdot}$
on the corank~$\fc_1$-subbundle $\cN^{\fc_1}\!\subset\!\cN_XY$ and determines 
an equivalence relation on the collection of all 
tubular neighborhood identifications for $Y\!\subset\!X$ (over $U_Y\!=\!X$).
We show in Subsection~\ref{augOpenSet_subsub} that the isomorphism class of 
an $\cN^{\fc_1}$-augmented blowup \hbox{$\pi\!:\BL_{\cN^{\fc_1}}\!X\!\lra\!X$} 
is determined by a subbundle $\cN^{\fc_1}\!\subset\!\cN_XY$
of corank $\fc_1\!\in\![\fc\!-\!1]$, a conformal class of inner-products on~$\cN^{\fc_1}$,
and an $(\cN^{\fc_1},\lr{\cdot,\cdot}\!)$-equivalence class of tubular neighborhood identifications
for $Y\!\subset\!X$ with respect to any inner-product~$\lr{\cdot,\cdot}$ in the conformal class.

\subsubsection{Open sets description}\label{augOpenSet_subsub}
Continuing with the setup and notation of Subsection~\ref{glaugbl_subsubs}, let
\begin{gather*}
\wt{W}_{Y;1}=\big\{(\ell,v)\!\in\!\wt{W}_Y\!:\ell\!\not\subset\!\cN^{\fc_1}\big\}
\subset\ga_{\cN_XY},\\
\wt{W}_{Y;2}=\big\{\!\big(\ell,(v,c)\!\big)\!\otimes\!(\ell,w)\!\in\!
\ga_{\cN'\oplus\tau_Y^1}^{\R}\!\otimes_{\R}\!\pi'^*\cN^{\fc_1}\!:
c\lr{w,w}v\!+\!cw\!\in\!W_Y\big\}.
\end{gather*}
Define
\BE{OpenSetpidfn_e0}
\wt\pi\!:\wt{W}_{Y;1}\!\sqcup\!\wt{W}_{Y;2}\lra X, \quad
\wt\pi(\wt{x})=\begin{cases}
\Psi_Y(v),
&\hbox{if}~\wt{x}\!\equiv\!(\ell,v)\!\in\!\wt{W}_{Y;1};\\
\Psi_Y\big(c\lr{w,w}v\!+\!cw\big),
&\hbox{if}~\wt{x}\!\equiv\!\big(\ell,(v,c)\!\big)\!\otimes\!(\ell,w)\!\in\!\wt{W}_{Y;2}.
\end{cases}\EE
By the construction in Subsection~\ref{glaugbl_subsubs}, 
\begin{gather*}
\BL_{\cN^{\fc_1}}\!X=\big(\!(X\!-\!Y)\!\sqcup\!\wt{W}_{Y;1}\!\sqcup\!\wt{W}_{Y;2}\big)\!\big/\!\sim, \\
\begin{split}
\wt{W}_{Y;2}|_{\R\P(\cN'\oplus\tau_Y^1)-\R\P(\{0\}\oplus\tau_Y^1)}
\!-\!\R\P(\cN'\!\oplus\!\tau_Y^1)\ni\big(\ell,(v,c)\!\big)\!\otimes\!(\ell,w)
&\sim\big([\lr{w,w}v\!+\!w],c\lr{w,w}v\!+\!cw\big)\in\wt{W}_{Y;1},
\end{split}\\
\big(\wt{W}_{Y;1}\!-\!\R\P(\cN_XY)\!\big) 
\!\sqcup\!\big(\wt{W}_{Y;2}|_{\R\P(\cN'\oplus\tau_Y^1)-\R\P(\cN'\oplus\tau_Y^0)}
\!-\!\R\P(\cN'\!\oplus\!\tau_Y^1)\!\big)\ni\wt{x} 
\sim\wt\pi(\wt{x})\in X\!-\!Y.
\end{gather*}
The blowdown map is given by
\BE{OpenSetpidfn_e}\pi\!:\BL_{\cN^{\fc_1}}\!X\lra X, \qquad 
\pi\big([\wt{x}]\big)=\begin{cases}\wt{x},&\hbox{if}~\wt{x}\!\in\!X\!-\!Y;\\
\wt\pi(\wt{x}),&\hbox{if}~\wt{x}\!\in\!\wt{W}_{Y;1}\!\sqcup\!\wt{W}_{Y;2}.
\end{cases}\EE

\vspace{.2in}

\noindent
Suppose $\lr{\cdot,\cdot}'$ is an inner-product on~$\cN^{\fc_1}$ conformally equivalent
to~$\lr{\cdot,\cdot}$, i.e.~$\lr{\cdot,\cdot}'\!=\!\rho^{-1}\lr{\cdot,\cdot}$
for some $\rho\!\in\!C^{\i}(Y;\R^+)$.
Define 
$$\wt{W}_{Y;2}'\subset\ga_{\cN'\oplus\tau_Y^1}^{\R}\!\otimes_{\R}\!\pi'^*\cN^{\fc_1}
\qquad\hbox{and}\qquad
\wt\pi'\!:\wt{W}_{Y;1}\!\sqcup\!\wt{W}_{Y;2}'\lra X$$
as $\wt{W}_{Y;2}$ and $\wt\pi$ in~\eref{OpenSetpidfn_e0} and just above 
with~$\lr{\cdot,\cdot}$ replaced by~$\lr{\cdot,\cdot}'$.
The~blowup 
$$\pi'\!:\big(\BL_{\cN^{\fc_1}}\!X\big)'\lra X$$
constructed as in Subsection~\ref{glaugbl_subsubs} with~$\lr{\cdot,\cdot}$ 
replaced by~$\lr{\cdot,\cdot}'$ is described by~\eref{OpenSetpidfn_e} and
the quotient just above
with~$\wt{W}_{Y;2}$ and~$\lr{\cdot,\cdot}$ replaced by~$\wt{W}_{Y;2}'$ and~$\lr{\cdot,\cdot}'$, 
respectively.
The~map
\begin{gather*}
\wt\Psi\!:\BL_{\cN^{\fc_1}}\!X\lra \big(\BL_{\cN^{\fc_1}}\!X\big)',\\
 \wt\Psi\big([\wt{x}]\big)=\begin{cases}[\wt{x}],&\hbox{if}~
\wt{x}\!\in\!(X\!-\!Y)\!\sqcup\!\wt{W}_{Y;1};\\
\big[([\rho v,c],(\rho v',c')\!)\!\otimes\!([\rho v,c],w)\big],&\hbox{if}~
\wt{x}\!=\!([v,c],(v',c')\!)\!\otimes\!([v,c],w)\!\in\!\wt{W}_{Y;2};
\end{cases}
\end{gather*}
is then a well-defined diffeomorphism so that $\pi\!=\!\pi'\!\circ\!\wt\Psi$.\\

\noindent
Suppose instead that $\cN''\!\subset\!\cN_XY$ is another subbundle complementary to~$\cN^{\fc_1}$.
Thus,
$$\cN''=\big\{v\!+\!\al(v)\!:v\!\in\!\cN'\big\}$$
for some $\al\!\in\!\Ga(Y;\Hom(\cN',\cN^{\fc_1})\!)$ and the homomorphism
$$\wt\al\!:\cN'\lra\cN'', \qquad \wt\al(v)=v\!+\!\al(v),$$
is an isomorphism.
Define 
$$\lr{\cdot,\cdot}_{\al}\!:\cN'\!\otimes_{\R}\!\cN^{\fc_1}\lra\R^{\ge0},\quad
\lr{v,w}_{\al}=1\!-\!2\lr{\al(v),w}\!+\!\lr{\al(v),\al(v)\!}\lr{w,w}\,.$$
If $\lr{v,w}_{\al}\!=\!0$, then $\lr{\al(v),w}\!=\!1$.
Let $\pi''\!:\R\P(\cN''\!\oplus\!\tau_Y^1)\!\lra\!Y$
denote the bundle projections.\\

\noindent
The image~of
$$\wt{W}_{Y;2}^{\al}\equiv\big\{\!\big(\ell,(v,c)\!\big)\!\otimes\!(\ell,w)\!\in\!\wt{W}_{Y;2}\!:
\lr{\al(v),w}\!\neq\!1\big\}$$
in~$\BL_{\cN^{\fc_1}}\!X$ under the quotient map is open.
The complement of this image is contained in the image of $(X\!-\!Y)\!\sqcup\!\wt{W}_{Y;1}$.
Thus, $\wt{W}_{Y;2}$ in~\eref{OpenSetpidfn_e} and just above can be replaced 
by~$\wt{W}_{Y;2}^{\al}$.
Define
$$\wt{W}_{Y;2}''\subset\ga_{\cN''\oplus\tau_Y^1}^{\R}\!\otimes_{\R}\!\pi''^*\cN^{\fc_1}
\qquad\hbox{and}\qquad
\wt\pi''\!:\wt{W}_{Y;1}\!\sqcup\!\wt{W}_{Y;2}''\lra X$$
as $\wt{W}_{Y;2}$ and $\wt\pi$ in~\eref{OpenSetpidfn_e0} and just above
with~$\cN'$ and~$\pi'$ replaced by~$\cN''$ and~$\pi''$.
The~map
\begin{gather*}
\wt\phi_2\!: \wt{W}_{Y;2}^{\al}\lra\wt{W}_{Y;2}''\subset
\ga_{\cN''\oplus\tau_Y^1}^{\R}\!\otimes_{\R}\!\pi''^*\cN^{\fc_1},\\
\wt\phi_2\big(\!([v,c],(v',c')\!)\!\otimes\!([v,c],w)\!\big)
=\bigg(\!\R\Big(\frac{\wt\al(v)}{\lr{v',w}_{\al}},c\!\Big),
\Big(\frac{\wt\al(v')}{\lr{v',w}_{\al}},c'\!\Big)\!\!\bigg)
\!\otimes\!\bigg(\!\R\Big(\frac{\wt\al(v)}{\lr{v',w}_{\al}},c\!\Big),
w\!-\!\lr{w,w}\al(v')\!\!\bigg),
\end{gather*}
is well-defined and satisfies $\wt\pi\!=\!\wt\pi''\!\circ\!\wt\phi_2$ on~$\wt{W}_{Y;2}^{\al}$.
The~blowup 
$$\pi''\!:\big(\BL_{\cN^{\fc_1}}\!X\big)''\lra X$$
constructed as in Subsection~\ref{glaugbl_subsubs} with~$\cN'$  replaced by~$\cN''$ 
is described by~\eref{OpenSetpidfn_e} and the quotient just above
with~$\wt{W}_{Y;2}$ replaced by~$\wt{W}_{Y;2}''$.
The~map
$$\wt\Psi\!:\BL_{\cN^{\fc_1}}\!X\lra \big(\BL_{\cN^{\fc_1}}\!X\big)'',\qquad
 \wt\Psi\big([\wt{x}]\big)=\begin{cases}[\wt{x}],&\hbox{if}~\wt{x}\!\in\!
(X\!-\!Y)\!\sqcup\!\wt{W}_{Y;1};\\
\big[\wt\phi_2(\wt{x})\big],&\hbox{if}~\wt{x}\!\in\!\wt{W}_{Y;2}^{\al};
\end{cases}$$
is thus a well-defined diffeomorphism so that $\pi\!=\!\pi'\!\circ\!\wt\Psi$.\\

\noindent 
Suppose $(W_Y',\Psi_Y')$ is another tubular neighborhood identification for $Y\!\subset\!X$
with \hbox{$\Psi_Y'(W_Y')\!\subset\!\Psi_Y(W_Y)$} and
$h,f$ are in~\eref{augatlcond_e0} with $\Psi_1\!=\!\Psi_Y$, $\Psi_2\!=\!\Psi_Y'$, 
and $W_{\cap}\!=\!W_Y'$.
Let
\begin{gather*}
h_{11}\in\Ga\big(W_Y',\Hom_{\bF}\big(\pi_{\cN_XY}^*\cN',
\Psi_Y'^*\Psi_Y^{-1*}\pi_{\cN_XY}^*\cN'\big)\!\big),
h_{21}\in\Ga\big(W_Y',\Hom_{\bF}\big(\pi_{\cN_XY}^*\cN',
\Psi_Y'^*\Psi_Y^{-1*}\pi_{\cN_XY}^*\cN^{\fc_1}\big)\!\big),\\
\hbox{and}\qquad
h_{22}\in\Ga\big(W_Y',\Hom_{\bF}\big(\pi_{\cN_XY}^*\cN^{\fc_1},
\Psi_Y'^*\Psi_Y^{-1*}\pi_{\cN_XY}^*\cN^{\fc_1}\big)\!\big)
\end{gather*}
be the components of $h$ with respect to the decomposition $\cN_XY\!=\!\cN'\!\oplus\!\cN^{\fc_1}$.
Define 
$$\wt\pi'\!:\wt{W}_{Y;1}'\!\sqcup\!\wt{W}_{Y;2}'\lra X$$
as~$\wt\pi$ in~\eref{OpenSetpidfn_e0} and just above with~$W_Y$ replaced by~$W_Y'$.\\

\noindent
By~\eref{augatlcond_e0}, $h_{22}(v)$ is an isomorphism for every $v\!\in\!W_Y'$.
By~\eref{TNIprop_e}, $h_{11}(v)$ is an isomorphism for every~$v$ in 
a neighborhood $W_Y'^{\circ}\!\subset\!W_Y'$ of~$Y$.
Define \hbox{$\wt{W}_{Y;1}'^{\circ}\!\subset\!\wt{W}_{Y;1}'$} as above~\eref{OpenSetpidfn_e0} 
with~$W_Y$ replaced by~$W_Y'^{\circ}$.
Similarly to~\eref{wtvphalalprdfn_e0}, the~map
\BE{wtvphalalprdfn_e4}\wt\phi_1\!:\wt{W}_{Y;1}'^{\circ}\lra\ga_{\cN_XY}, \quad 
\wt\phi_1(\ell,v)=\big(\!\{h(v)\!\}(\ell),\{h(v)\!\}(v)\!\big),\EE
is then well-defined and smooth.
By~\eref{TNIequiv_e}, $\wt\phi_1(\wt{W}_{Y;1}'^{\circ})\!\subset\!\wt{W}_{Y;1}$
and $\wt\pi'\!=\!\wt\pi\!\circ\!\wt\phi_1$ on~$\wt{W}_{Y;1}'^{\circ}$.\\

\noindent
Let $|w|^2\!=\!\lr{w,w}$ for $w\!\in\!\cN^{\fc_1}$.
The smooth function 
\begin{gather*}
F\!:\wt{W}_{Y;2}'\lra \R,\\
\begin{split}
F\big(\!(\ell,(v,c)\!)\!\otimes\!(\ell,w)\!\big)=
f\big(c|w|^2v\!+\!cw\big)\!
+\!2\blr{\!\big\{h_{21}(c|w|^2v\!+\!cw)\!\big\}(v),
\big\{h_{22}(c|w|^2v\!+\!cw)\!\big\}(w)\!}&\\
+|w|^2\big|\big\{h_{21}(c|w|^2v\!+\!cw)\!\big\}(v)\big|^2&,
\end{split}\end{gather*}
is $\R^+$-valued on $\wt{W}_{Y;2}'|_{\R\P(\{0\}\oplus\tau_Y^1)}\!\cup\!\R\P(\cN'\!\oplus\!\tau_Y^1)$.
We can thus choose a neighborhood \hbox{$\wt{W}_{Y;2}'^{\circ}\!\subset\!\wt{W}_{Y;2}'$} of this set 
so that the~map
\begin{gather*}
\phi_2\!:\wt{W}_{Y;2}'^{\circ}\lra\R\P\big(\cN'\!\oplus\!\tau_Y^1\big), \\
\begin{split}
&\phi_2\big(\!([v,c],(v',c')\!)\!\otimes\!([v,c],w)\!\big)=
\R\big(\big\{h_{11}(c'|w|^2v'\!+\!c'w)\!\big\}(v),
F\big(\!\big([v,c],(v',c')\!\big)\!\otimes\!([v,c],w)\!\big)c\big),
\end{split}\end{gather*}
is well-defined and smooth.
So is then the~map
\begin{gather}\label{wtphi2dfn_e}
\wt\phi_2\!:\wt{W}_{Y;2}'^{\circ}\lra\ga_{\cN'\oplus\tau_Y^1}^{\R}\!\otimes_{\R}\!\pi'^*\cN^{\fc_1},\\
\notag\begin{split}
&\wt\phi_2\big(\!\big(\ell,(v,c)\!\big)\!\otimes\!(\ell,w)\!\big)
=\bigg(\!\phi_2\big(\!(\ell,(v,c)\!)\!\otimes\!(\ell,w)\!\big),
\Big(\frac{\big\{h_{11}(c|w|^2v\!+\!cw)\!\big\}(v)}{F\big(\!(\ell,(v,c)\!)\!\otimes\!(\ell,w)\!\big)},
c\!\Big)\!\!\bigg)\\
&\hspace{1.5in}\otimes\!\Big(\phi_2\big(\!(\ell,(v,c)\!)\!\otimes\!(\ell,w)\!\big),
\big\{\!h_{22}(c|w|^2v\!+\!cw)\!\big\}(w)\!+\!
|w|^2\big\{\!h_{21}(c|w|^2v\!+\!cw)\!\big\}(v)\!\Big).
\end{split}\end{gather}
By~\eref{augatlcond_e0} and~\eref{TNIequiv_e} with $v$ replaced by~$c|w|^2v\!+\!cw$, 
$\wt\phi_2(\wt{W}_{Y;2}'^{\circ})\!\subset\!\wt{W}_{Y;2}$
and $\wt\pi'\!=\!\wt\pi\!\circ\!\wt\phi_2$ on~$\wt{W}_{Y;2}'^{\circ}$.\\

\noindent
We can assume that $\wt\pi'(\wt{W}_{Y;2}'^{\circ})\!\subset\!W_Y'^{\circ}$.
The $\cN^{\fc_1}$-augmented blowup 
$$\pi'\!:(\BL_{\cN^{\fc_1}}\!X)'\lra X$$ 
of~$X$ along~$Y$ constructed as in Subsection~\ref{glaugbl_subsubs} with~$(W_Y,\Psi_Y)$ 
replaced by~$(W_Y',\Psi_Y')$ can then be described as in~\eref{OpenSetpidfn_e} and just above with 
$\wt{W}_{Y;1},\wt{W}_{Y;2}$ replaced by~$\wt{W}_{Y;1}'^{\circ},\wt{W}_{Y;2}'^{\circ}$.
The~map
$$\wt\Psi\!:\big(\BL_{\cN^{\fc_1}}\!X\big)'\lra \BL_{\cN^{\fc_1}}\!X,\qquad
 \wt\Psi\big([\wt{x}]\big)=\begin{cases}[\wt{x}],&\hbox{if}~\wt{x}\!\in\!X\!-\!Y;\\
\big[\wt\phi_i(\wt{x})\big],&\hbox{if}~\wt{x}\!\in\!\wt{W}_{Y;i}'^{\circ},\,i\!=\!1,2;
\end{cases}$$
is thus a well-defined diffeomorphism so that $\pi'\!=\!\pi\!\circ\!\wt\Psi$.

\subsection{Local construction}
\label{locaugbl_subs}

\noindent
Let $\fc\!\in\!\Z^+$, $\fc_1\!\in\![\fc\!-\!1]$, 
$\fc_2\!=\!\fc\!-\!\fc_1$, $\dbsqbr{\fc_1}\!=\!\{0\}\!\sqcup\![\fc_1]$,
$$\ga_{\fc;\fc_1}^1=\big\{\!\big([r_1,\ldots,r_{\fc}],v\big)\!\in\!\ga^{\R}_{\fc}\!:
(r_1,\ldots,r_{\fc_1})\!\neq\!0\big\}\lra \R\P^{\fc-1},
~~\hbox{and}~~
\ga_{\fc;\fc_1}^2=\big(\ga_{\fc_1+1}^{\R}\big)^{\!\oplus\fc_2}
\lra\R\P^{\fc_1}.$$
Let $\pi_{\fc;\fc_1}^1\!:\ga_{\fc;\fc_1}^1\!\lra\!\R^{\fc}$ be the restriction of 
the projection to the second coordinate on~$\R\P^{\fc-1}\!\times\!\R^{\fc}$.
For any $(\ell,(v_{ij})_{i\in\dbsqbr{\fc_1},j\in[\fc_2]})\!\in\!\ga_{\fc;\fc_1}^2$
with $\ell\!=\![(r_i)_{i\in\dbsqbr{\fc_1}}]$, 
there exists $\la\!\in\!\R^{\fc_2}$ so that 
$$v_{ij}=r_i\la_j \qquad\forall~i\!\in\!\dbsqbr{\fc_1},\,j\!\in\![\fc_2].$$
Thus, the~maps 
\begin{gather*}
\pi_{\fc;\fc_1}^2\!:\ga_{\fc;\fc_1}^2\lra\R^{\fc},~~
\pi_{\fc;\fc_1}^2\big(\ell,(r_i\la_j)_{i\in\dbsqbr{\fc_1},j\in[\fc_2]}\big)
=\bigg(\!r_0\bigg(\sum_{j=1}^{\fc_2}\!\la_j^2\bigg)
\big(r_1,\ldots,r_{\fc_1}\big),r_0\big(\la_1,\ldots,\la_{\fc_2}\big)\!\!\bigg),\\
\phi_{\fc;\fc_1;0}^{12}\!:\ga_{\fc;\fc_1}^{2;1}\!\equiv\!\big\{\!
\big(\ell,(v_{ij})_{i\in\dbsqbr{\fc_1},j\in[\fc_2]}\big)
\!\in\!\ga_{\fc;\fc_1}^2\!:
(v_{ij})_{i\in[\fc_1],j\in[\fc_2]}\!\neq\!0\!\in\!\R^{\fc_1\fc_2}\big\}\lra \R\P^{\fc-1},\\
\phi_{\fc;\fc_1;0}^{12}
\big(\ell,(r_i\la_j)_{i\in\dbsqbr{\fc_1},j\in[\fc_2]}\big)
=\bigg[\!\bigg(\sum_{j=1}^{\fc_2}\!\la_j^2\bigg)
\big(r_1,\ldots,r_{\fc_1}\big),\big(\la_1,\ldots,\la_{\fc_2}\big)\!\bigg],
\end{gather*}
are well-defined and smooth.\\

\noindent
The~map
$$\wt\phi_{\fc;\fc_1}^{12}\!\equiv\!
\big(\phi_{\fc;\fc_1;0}^{12},\pi_{\fc;\fc_1}^2\big)
\!:\ga_{\fc;\fc_1}^{2;1}\lra 
\ga_{\fc;\fc_1}^{1;2}
\!\equiv\!\big\{\!
\big([(r_j)_{j\in[\fc]}],v\big)\!\in\!\ga^1_{\fc;\fc_1}\!:
(r_j)_{j\in[\fc]-[\fc_1]}\!\neq\!0\big\}$$
is a well-defined diffeomorphism onto an open subspace of~$\ga_{\fc;\fc_1}^1$.
It satisfies
\BE{pifcfc1dfn_e}
\pi_{\fc;\fc_1}^2\big|_{\ga_{\fc;\fc_1}^{2;1}}=
\pi_{\fc;\fc_1}^1\!\circ\!\wt\phi_{\fc;\fc_1}^{12}.\EE
The space
$$\ga_{\fc;\fc_1}^{\au}=\big(\ga_{\fc;\fc_1}^1\!\sqcup\!\ga_{\fc;\fc_1}^2
\big)\!\big/\!\sim,\quad 
\ga_{\fc;\fc_1}^1\ni \wt\phi_{\fc;\fc_1}^{12}(\wt{x})
\sim \wt{x}\in \ga_{\fc;\fc_1}^{2;1}\subset\ga_{\fc;\fc_1}^2,$$
is a smooth manifold.
By~\eref{pifcfc1dfn_e}, the map
$$\pi\!:\ga_{\fc;\fc_1}^{\au}\lra\R^{\fc}, \qquad
\pi\big([\wt{r}]\big)=\begin{cases} 
\pi_{\fc;\fc_1}^1(\wt{r}),
&\hbox{if}~\wt{r}\!\in\!\ga_{\fc;\fc_1}^1;\\
\pi_{\fc;\fc_1}^2(\wt{r}),
&\hbox{if}~\wt{r}\!\in\!\ga_{\fc;\fc_1}^2;
\end{cases}$$
is well-defined and smooth.\\

\noindent
Suppose $Y\!\subset\!X$ is a closed submanifold of real codimension~$\fc$ and dimension~$m$
as before.
For a coordinate chart 
$$\vph\!\equiv\!(\vph_1,\ldots,\vph_{\fc+m})\!:U\lra\R^{\fc}\!\times\!\R^m$$
for~$Y$ in~$X$, the subspaces
\begin{equation*}\begin{split}
\BLauk_Y\vph&\equiv\big\{\!(\wt{r},x)\!\in\!\ga^k_{\fc;\fc_1}\!\times\!U\!:
\pi_{\fc;\fc_1}^k(\wt{r})\!=\!\big(\vph_1(x),\ldots,\vph_{\fc}(x)\!\big)\!\in\!\R^{\fc}\big\}
~~\hbox{with}~~k=1,2 \quad\hbox{and}\\
\BLau_Y\vph&\equiv \big\{\!(\wt{r},x)\!\in\!\ga^{\au}_{\fc;\fc_1}\!\times\!U\!:
\pi(\wt{r})\!=\!\big(\vph_1(x),\ldots,\vph_{\fc}(x)\!\big)
\!\in\!\R^{\fc}\big\} 
\end{split}\end{equation*}
are closed submanifolds of~$\ga^k_{\fc;\fc_1}\!\times\!U$
 and $\ga^{\au}_{\fc;\fc_1}\!\times\!U$, respectively.
Let
$$\BLauba_Y\vph\equiv\big\{\!(\wt{r},x)\!\in\!\BLaub_Y\vph\!:
\wt{r}\!\in\!\ga^{2;1}_{\fc;\fc_1}\big\}.$$
We denote~by
$$\pi^k_U\!:\BLauk_Y\vph\lra U \quad\hbox{and}\quad
\pi_U\!:\BLau_Y\vph\lra U$$
the projection maps.
The restriction
$$\pi_U\!:\BLau_Y\vph\!-\!\pi_U^{-1}(Y) \lra U\!-\!Y$$
is a diffeomorphism.\\

\noindent
We call an $\R$-atlas 
$\{\vph_{\al}\!:U_{\al}\!\lra\!\R^{\fc}\!\times\!\R^m\}_{\al\in\cI}$ 
for~$Y$ in~$X$ \sf{$\fc_1$-augmented} if for all $\al,\al'\!\in\!\cI$ there exist
$h_{\al\al'}$ as in~\eref{halalprdfn_e} and
\hbox{$f_{\al\al'}\!\in\!C^{\i}(\vph_{\al'}(U_{\al}\!\cap\!U_{\al'});\R^+)$} such~that
\BE{augatlcond_e}\begin{split}
& \big\{h_{\al\al'}(x)\!\big\}\big(0^{\fc_1}\!\times\!\R^{\fc_2}\big)
\subset 0^{\fc_1}\!\times\!\R^{\fc_2}\\
&\big|\big\{h_{\al\al'}(x)\!\big\}\!(0,r)\big|^2
=f_{\al\al'}(x)|r|^2
\end{split}
\qquad\forall~x\!\in\!\vph_{\al'}\big(U_{\al}\!\cap\!U_{\al'}\big),\,
r\!\in\!\R^{\fc_2}.\EE
By the first condition in~\eref{augatlcond_e},
the images of $0^{\fc_1}\!\times\!\R^{\fc_2}\!\times\!(U_{\al}\!\cap\!Y)$
under the vector bundle isomorphisms~$\Phi_{\al}$ in~\eref{NXYisom_e}
determine a corank~$\fc_1$ subbundle $\cN^{\fc_1}\!\subset\!\cN_XY$.
We will then call a $\fc_1$-augmented atlas 
$\{\vph_{\al}\}_{\al\in\cI}$ \sf{$\cN^{\fc_1}$-augmented}.\\

\noindent
By the first condition in~\eref{augatlcond_e},
there exists a neighborhood~$U_{\al\al'}$ of~$U_{\al}\!\cap\!U_{\al'}\!\cap\!Y$ 
in~$U_{\al}\!\cap\!U_{\al'}$ so that the composition
$$\R^{\fc_1}\!\times\!0^{\fc_2}\xlra{h_{\al\al'}(x)}\R^{\fc_1}\!\times\!\R^{\fc_2}\lra\R^{\fc_1}$$
is invertible for every $x\!\in\!\vph_{\al'}(U_{\al\al'})$.
Let
$$\wt{U}_{\al\al'}^1=\big\{\pi_{U_{\al'}}^1\big\}^{-1}\big(U_{\al\al'}\big)
\subset 
\BLaua_Y\big(\vph_{\al'}|_{U_{\al}\cap U_{\al'}}\big).$$
Similarly to~\eref{wtvphalalprdfn_e4}, the~map
\begin{gather*}
\wt\vph_{\al\al'}^1\!:\wt{U}_{\al\al'}^1
\lra \BLaua_Y\big(\vph_{\al}|_{U_{\al\al'}}\big), \\
\wt\vph_{\al\al'}^1\big(\!(\ell,v),x\big)
=\big( \big(\!\big\{\!h_{\al\al'}(\vph_{\al'}(x)\!)\!\big\}(\ell),
\big\{\!h_{\al\al'}(\vph_{\al'}(x)\!)\!\big\}(v)\!\big),x\big),
\end{gather*}
is well-defined and smooth.
It satisfies $\pi_{U_{\al'}}^1\!=\!\pi_{U_{\al}}^1\!\circ\!\wt\vph_{\al\al'}^1$
on~$\wt{U}_{\al\al'}^1$.
Similarly to~\eref{wtphi2dfn_e}, $h_{\al\al'}$ and $f_{\al\al'}$ determine a diffeomorphism
$$\wt\vph_{\al\al'}^2\!:\wt{U}_{\al\al'}^2
\lra \BLaub_Y\big(\vph_{\al}|_{U_{\al\al'}}\big)$$
from an open neighborhood of 
$\BLaub_Y(\vph_{\al'}|_{U_{\al\al'}})\!-\!\BLauba_Y(\vph_{\al'}|_{U_{\al\al'}})$
in~$\BLaub_Y(\vph_{\al'}|_{U_{\al\al'}})$ onto an open subset 
of~$\BLaub_Y(\vph_{\al}|_{U_{\al\al'}})$; 
see \cite[Section~3.2]{RDMbl} for details.
It satisfies $\pi_{U_{\al'}}^2\!=\!\pi_{U_{\al}}^2\!\circ\!\wt\vph_{\al\al'}^2$
on~$\wt{U}_{\al\al'}^2$.\\

\noindent
By the above, the~map
\begin{gather*}
\wt\vph_{\al\al'}\!:
\BLau_Y\big(\vph_{\al'}|_{U_{\al}\cap U_{\al'}}\big)
\lra \BLau_Y\big(\vph_{\al}|_{U_{\al}\cap U_{\al'}}\big),\\
\wt\vph_{\al\al'}\big([\wt{x}']\big)=
\begin{cases}\big[\wt\vph_{\al\al'}^k(\wt{x}')\big],\!\!
&\hbox{if}~\wt{x}'\!\in\!\wt{U}_{\al\al'}^k,\,k\!=\!1,2;\\
[\wt{x}],&
\hbox{if}~\pi_{U_{\al'}}([\wt{x}'])\!=\!\pi_{U_{\al}}([\wt{x}])\!\not\in\!Y;
\end{cases}
\end{gather*}
is then a well-defined diffeomorphism that satisfies
$$\pi_{U_{\al'}}\big|_{\BLau_Y\big(\vph_{\al'}|_{U_{\al}\cap U_{\al'}}\big)}
=\pi_{U_{\al}}\!\circ\!\wt\vph_{\al\al'}.$$
By the uniqueness of continuous extensions and the cocycle condition
for the overlap maps~$\vph_{\al\al'}$,
the maps~$\wt\vph_{\al\al'}$ satisfy the cocycle condition~\eref{cocyclecond_e}
with~$\bF\!=\!\fc_1$.
For a $\fc_1$-augmented atlas $\{\vph_{\al}\}_{\al\in\cI}$ for~$Y$ in~$X$,
we define the corresponding $\fc_1$-augmented blowup
\BE{blowmapaugdfn_e}\pi\!:\BL_Y^{\fc_1}X
\!\equiv\!\Big(\!\!
(X\!-\!Y)\!\sqcup\!\bigsqcup_{\al\in\cI}\!\!\BLau_Y\vph_{\al}\!\Big)\!\!\Big/\!\!\!\sim
\lra X\EE
as in~\eref{blowmapdfn_e} and just above with~$\bF$ replaced by~$\fc_1$.\\

\noindent
We call two $\fc_1$-augmented atlases for~$Y$ in~$X$
 \sf{$\fc_1$-augmented equivalent} if their union is still  
a $\fc_1$-augmented atlas for~$Y$ in~$X$.
Two such atlases determine the same corank~$\fc_1$-subbundle \hbox{$\cN^{\fc_1}\!\subset\!\cN_XY$}
and isomorphic $\fc_1$-augmented blowups of~$X$ along~$Y$.

\subsection{Equivalence}
\label{equivaugbl_subs}

\noindent
We now refine the two constructions of Section~\ref{equivRCbl_subs} to incorporate
the additional data used in the augmented blowup constructions
of Sections~\ref{glaugbl_subs} and~\ref{locaugbl_subs}.
The following statement will be used to conclude that the augmented blowups 
we obtain as a result are isomorphic.

\begin{lmm}[{\cite[Lemma~3.2]{RDMbl}}]\label{AugBl_lmm}
Suppose $X$, $Y$, $\fc$, and $\pi$ are as in Lemma~\ref{StanBl_lmm} with $\bF\!=\!\R$,
$\fc_1\!\in\![\fc\!-\!1]$, and
$\{\vph_{\al}\!\equiv\!(\vph_{\al;1},\ldots,\vph_{\al;\fc+m})\!:
U_{\al}\!\lra\!\R^{\fc}\!\times\!\R^m\}_{\al\in\cI}$ 
is a \hbox{$\fc_1$-augmented} atlas of charts for~$Y$ in~$X$.
Let \hbox{$\pi'\!:\BLau_YX\!\lra\!X$} be the $\fc_1$-augmented blowup of~$X$ along~$Y$
determined by this atlas as in Section~\ref{locaugbl_subs}.
If there exists a collection 
$$\big\{\big(\wt\phi_{\al;i}^k\!\equiv\!
(\wt\phi_{\al;i;1}^k,\ldots,\wt\phi_{\al;i;\fc+m}^k)
\!:\wt{U}_{\al;i}^k\!\lra\!\R^{\fc+m}\big)\!:
k\!=\!1,2,\,i\!\in\![\fc_1]\,\hbox{if}\,k\!=\!1,\,
i\!\in\!\dbsqbr{\fc_1}\,\hbox{if}\,k\!=\!2\big\}_{\al\in\cI}$$
of charts on~$\wt{X}$ so that 
\begin{gather*}
\pi^{-1}(U_{\al})=\bigcup_{i=1}^{\fc_1}\!\wt{U}_{\al;i}^1\cup
\bigcup_{i=0}^{\fc_1}\!\wt{U}_{\al;i}^2, \quad
\wt\phi_{\al;i;j}^k=\vph_{\al;j}\!\circ\!\pi|_{\wt{U}_{\al;i}^k}
~~\hbox{if}~j\!>\!\fc~\hbox{or}~(k,j)\!=\!(1,i),\\
\vph_{\al;j}\!\circ\!\pi|_{\wt{U}_{\al;i}^1}=
\wt\phi_{\al;i;j}^1\!\cdot\!\{\vph_{\al;i}\!\circ\!\pi|_{\wt{U}_{\al;i}^1}\}
~\hbox{if}\,j\!\in\![\fc]\!-\!\{i\},\\
%\end{gather*}
%\begin{gather*}
\vph_{\al;j}\!\circ\!\pi|_{\wt{U}_{\al;i}^2}=
\wt\phi_{\al;i;j}^2\!\cdot\!
\begin{cases}1,&\hbox{if}~i\!=\!0,\,j\!\in\![\fc]\!-\![\fc_1];\\
\wt\phi_{\al;i;1}^2,&\hbox{if}~i\!\in\![\fc_1],\,j\!\in\![\fc]\!-\![\fc_1];
\end{cases}\\
\vph_{\al;j}\!\circ\!\pi|_{\wt{U}_{\al;i}^2}=
\bigg(\!\sum_{j'=\fc_1+1}^{\fc}\!\!\!\!(\wt\phi_{\al;i;j'}^2)^2\!\!\bigg)
\!\cdot\!
\begin{cases}\wt\phi_{\al;i;j}^2,&\hbox{if}~i\!=\!0,\,j\!\le\!\fc_1;\\
\wt\phi_{\al;i;1}^2\wt\phi_{\al;i;j}^2,&\hbox{if}~1\!\le\!i\!<\!j\!\le\!\fc_1;\\
\wt\phi_{\al;i;1}^2,&\hbox{if}~i\!=\!j;\\
\wt\phi_{\al;i;1}^2\wt\phi_{\al;i;j+1}^2,&\hbox{if}~i\!>\!j;
\end{cases}
\end{gather*} 
for every~$\al\!\in\!\cI$, then the maps $\pi$ and $\pi'$ are diffeomorphic.
\end{lmm}

\noindent
Suppose $Y\!\subset\!X$, $\fc$, $\fc_1$, $\cN^{\fc_1}\!\subset\!\cN_XY$, 
$\lr{\cdot,\cdot}$, and 
$(W_Y,\Psi_Y)$ are as in Subsection~\ref{glaugbl_subsubs}.
Let $m$ be the dimension of~$Y$, \hbox{$\fc_2\!=\!\fc\!-\!\fc_1$}, and
\hbox{$\pi\!:\wt{X}\!\lra\!X$}
be the corresponding augmented blowup of~$X$ along~$Y$.
Choose a collection~$\{\Phi_{\al}\}_{\al\in\cI}$ of trivializations  of~$\cN_XY$ as in~\eref{Phialdfn_e}
with 
$$\Phi_{\al}(\cN^{\fc_1})\subset 0^{\fc_1}\!\times\!\R^{\fc_2}\!\times\!\R^m,\quad
\big|\big(\Phi_{\al;\fc_1+1}(w),\ldots,\Phi_{\al;\fc}(w)\!\big)\!\big|^2
=\lr{w,w} ~~\forall\,w\!\in\!\cN^{\fc_1}\big|_{U_{Y;\al}}.$$
The corresponding transition maps~$h_{\al\al'}$ in~\eref{halalprequiv_e} then satisfy
$$\big\{\!h_{\al\al'}(x)\!\big\}(0^{\fc_1}\!\times\!\R^{\fc_2})=0^{\fc_1}\!\times\!\R^{\fc_2},
\quad
\big|\big\{\!h_{\al\al'}(x)\!\big\}(0,r)\big|^2=|r|^2
\quad\forall~x\!\in\!U_{Y;\al}\!\cap\!U_{Y;\al'},\,r\!\in\!\R^{\fc_2}.$$
The overlap map~$\vph_{\al\al'}$ between the coordinate charts~$\vph_{\al}$ and~$\vph_{\al'}$
in~\eref{equivchartdfn_e} now satisfies~\eref{augatlcond_e} with \hbox{$f_{\al\al'}\!=\!1$}.
The image of $0^{\fc_1}\!\times\!\R^{\fc_2}\!\times\!U_{Y;\al}$
under the vector bundle isomorphism in~\eref{NXYisom_e}
is the subbundle $\cN^{\fc_1}\!\subset\!\cN_XY$.
Thus, the collection $\{\vph_{\al}\}_{\al\in\cI}$ is an $\cN^{\fc_1}$-augmented
atlas for~$Y$ in~$X$.
It is immediate that 
$(\cN^{\fc_1},\lr{\cdot,\cdot}\!)$-equivalent tubular neighborhood identifications for $Y\!\subset\!X$ 
produce $\fc_1$-augmented atlases for $Y\!\subset\!X$ that are $\fc_1$-augmented equivalent.\\

\noindent
With the notation as in Subsection~\ref{augOpenSet_subsub} and $\al\!\in\!\cI$, let
\begin{equation*}\begin{split}
\wt{W}_{Y;1}\big|_{U_{Y;\al}}&=\big\{\!(\ell,v)\!\in\!\wt{W}_{Y;1}\!:
\ell\!\subset\!\cN_XY|_{U_{Y;\al}}\big\} \qquad\hbox{and}\\
\wt{W}_{Y;2}\big|_{U_{Y;\al}}&=\big\{\!\big(\ell,(v,c)\!\big)\!\otimes\!(\ell,w)\!\in\!\wt{W}_{Y;2}\!:
\ell\!\subset\!\cN_XY|_{U_{Y;\al}}\!\oplus\!\tau_{U_{Y;\al}}^1\big\}.
\end{split}\end{equation*}				
The union of these two open subsets of~$\wt{X}$ is the preimage of 
$U_{\al}\!\equiv\!\Psi_Y(W_Y|_{U_{Y;\al}})$ 
under the blowdown map \hbox{$\pi\!:\wt{X}\!\lra\!X$}.
For each $i\!\in\![\fc_1]$, the subspaces 
\begin{equation*}\begin{split}
\wt{U}_{\al;i}^1&\equiv\big\{\!\big([v],v'\big)\!\in\!\wt{W}_{Y;1}|_{U_{Y;\al}}\!:
\Phi_{\al;i}(v)\!\neq\!0\big\}\subset \wt{X} \qquad\hbox{and}\\
\wt{U}_{\al;i}^2&\equiv\big\{
\!\big([v,c],(v',c')\!\big)\!\otimes\!([v,c],w)\!\in\!\wt{W}_{Y;2}|_{U_{Y;\al}}\!:
\Phi_{\al;i}(v)\!\neq\!0\big\}\subset \wt{X}
\end{split}\end{equation*}
are open.
So is the subspace
$$\wt{U}_{\al;0}^2\equiv\big\{
\!\big([v,c],(v',c')\!\big)\!\otimes\!([v,c],w)\!\in\!\wt{W}_{Y;2}|_{U_{Y;\al}}\!:
c\!\neq\!0\big\}\subset \wt{X}.$$
The unions of the sets $\wt{U}_{\al;i}^1$ with $i\!\in\![\fc_1]$ and 
$\wt{U}_{\al;i}^2$ with $i\!\in\!\dbsqbr{\fc_1}$ are~$\wt{W}_{Y;1}\big|_{U_{Y;\al}}$
and~$\wt{W}_{Y;2}\big|_{U_{Y;\al}}$, respectively.
Below we define coordinate charts
$$\wt\phi_{\al;i}^k\!\equiv\!\big(\wt\phi_{\al;i;1}^k,\ldots,\wt\phi_{\al;i;\fc+m}^k\big)\!:
\wt{U}_{\al;i}^k\lra\R^{\fc+m}$$
on~$\wt{X}$ as in Lemma~\ref{AugBl_lmm}.\\

\noindent
For each $i\!\in\![\fc_1]$, let
$$\wt\phi_{\al;i;j}^1\big([v],v'\big)
=\begin{cases}\Phi_{\al;j}(v)/\Phi_{\al;i}(v),&\hbox{if}~j\!\in\![\fc]\!-\!\{i\};\\
\Phi_{\al;j}(v'),&\hbox{if}~j\!\in\!\{i\}\!\cup\!\big([\fc\!+\!m]\!-\![\fc]).
\end{cases}$$
We similarly define
$$\wt\phi_{\al;i;j}^2\big(\!\big([v,c],(v',c')\!\big)\!\otimes\!([v,c],w)\!\big)
=\begin{cases}c/\Phi_{\al;i}(v),&\hbox{if}~j\!=\!1;\\
\Phi_{\al;j-1}(v)/\Phi_{\al;i}(v),&\hbox{if}~j\!\in\![i]\!-\!\{1\};\\
\Phi_{\al;j}(v)/\Phi_{\al;i}(v),&\hbox{if}~j\!\in\![\fc_1]\!-\![i];\\
\Phi_{\al;i}(v')\Phi_{\al;j}(w),&\hbox{if}~j\!\in\![\fc]\!-\![\fc_1];\\
\Phi_{\al;j}(c'w),&\hbox{if}~j\!\in\![\fc\!+\!m]\!-\![\fc].
\end{cases}$$
Finally, let
$$\wt\phi_{\al;0;j}^2\big(\!\big([v,c],(v',c')\!\big)\!\otimes\!([v,c],w)\!\big)
=\begin{cases}\Phi_{\al;j}(v)/c,&\hbox{if}~j\!\in\![\fc_1];\\
\Phi_{\al;j}(c'w),&\hbox{if}~j\!\in\![\fc\!+\!m]\!-\![\fc_1].
\end{cases}$$
The collection $\{\wt\phi_{\al;i}^k\!:\wt{U}_{\al;i}^k\!\lra\!\R^{\fc+m}\}$ 
satisfies the conditions of Lemma~\ref{AugBl_lmm}.
Thus, the $\cN^{\fc_1}$-augmented blowup $\wt{X}\!\equiv\!\BL_{\cN^{\fc_1}}X$ obtained from 
the subbundle $\cN'\!\subset\!\cN_XY$, inner-product~$\lr{\cdot,\cdot}$ on~$\cN^{\fc_1}$,
and the tubular neighborhood identification~$(W_Y,\Psi_Y)$ for $Y\!\subset\!X$
via the construction of Subsection~\ref{glaugbl_subsubs} is isomorphic to 
the $\fc_1$-augmented blowup~$\BL_Y^{\fc_1}X$ obtained from the above $\fc_1$-augmented atlas
$\{\vph_{\al}\}_{\al\in\cI}$ cut up from~$(W_Y,\Psi_Y)$ via the construction
of Section~\ref{locaugbl_subs}.\\

\noindent
By Proposition~\ref{Loc2Glaug_prp} below, an $\cN^{\fc_1}$-augmented atlas 
$\{\vph_{\al}\}_{\al\in\cI}$ of charts  for~$Y$ in~$X$ can be assembled 
into a tubular neighborhood identification~$(W_Y,\Psi_Y)$ for $Y\!\subset\!X$
and an inner-product~$\lr{\cdot,\cdot}$ on~$\cN^{\fc_1}$
so that any $\fc_1$-augmented atlas cut out from~$(W_Y,\Psi_Y)$ is 
$\fc_1$-augmented equivalent to~$\{\vph_{\al}\}_{\al\in\cI}$ 
and thus the blowups obtained via the local construction
of Section~\ref{locaugbl_subs} from the two $\fc_1$-augmented atlases and 
the blowup obtained via the global construction of Subsection~\ref{glaugbl_subsubs} are isomorphic.
Furthermore, inner-products~$\lr{\cdot,\cdot}$ assembled from $\fc_1$-augmented equivalent atlases 
are conformally equivalent and tubular neighborhood identifications
assembled from such atlases are $(\cN^{\fc_1},\lr{\cdot,\cdot})$-equivalent.
This establishes a correspondence between the global and local blowup constructions
of Subsection~\ref{glaugbl_subsubs} and Section~\ref{locaugbl_subs}, respectively.

\begin{prp}\label{Loc2Glaug_prp}
Suppose $X$, $Y$, $\fc$, and $\fc_1$ are as in Lemma~\ref{AugBl_lmm},
\hbox{$\cN^{\fc_1}\!\subset\!\cN_XY$} is a subbundle of corank~$\fc_1$, and
$\{\vph_{\al}\}_{\al\in\cI}$ is an $\cN^{\fc_1}$-augmented atlas of charts for~$Y$ in~$X$.
There exist an inner-product~$\lr{\cdot,\cdot}$ on~$\cN^{\fc_1}$ and
a tubular neighborhood identification~$(W_Y,\Psi_Y)$ for~$Y\!\subset\!X$
such that any $\fc_1$-augmented atlas of charts cut out from~$(W_Y,\Psi_Y)$ and~$\lr{\cdot,\cdot}$
as above is $\fc_1$-augmented equivalent to~$\{\vph_{\al}\}_{\al\in\cI}$.
\end{prp}

\begin{proof} We will call a chart $\vph\!:U\!\lra\!\R^{\fc+m}$ for~$Y$ in~$X$
\sf{$\fc_1$-augmented compatible} with the $\cN^{\fc_1}$-augmented atlas $\{\vph_{\al}\}_{\al\in\cI}$
if the union of this atlas with $\{\vph\}$ is still an $\fc_1$-augmented atlas for~$Y$ in~$X$.
We will call a tubular neighborhood identification~$(W_Y,\Psi_Y)$ for $Y\!\subset\!X$
over an open subset $U_Y\!\subset\!Y$ and an inner-product~$\lr{\cdot,\cdot}$ on~$\cN^{\fc_1}$
\sf{$\fc_1$-augmented compatible} with $\{\vph_{\al}\}_{\al\in\cI}$
if any coordinate chart for $Y\!\subset\!X$ cut out from~$(W_Y,\Psi_Y)$ as above is 
$\fc_1$-augmented compatible with $\{\vph_{\al}\}_{\al\in\cI}$.\\

\noindent
For each $\al\!\in\!\cI$, let $U_{\al}$, $U_{Y;\al}$, $\Phi_{\al}$, and $(W_{\al},\Psi_{\al})$
be as in the proof of Proposition~\ref{Loc2GlRC_prp} and $\lr{\cdot,\cdot}_{\al}$
be the inner-product on~$\cN^{\fc_1}|_{U_{Y;\al}}$ induced from the standard inner-product 
on~$0^{\fc_1}\!\times\!\R^{\fc_2}$ via~$\Phi_{\al}$.
By the second statement in~\eref{augatlcond_e},
\BE{augatlcond_e7}\lr{\cdot,\cdot}_{\al}\big|_{U_{Y;\al}\cap U_{Y;\al'}}
=f_{\al\al'}\!\circ\!\vph_{\al'}\big|_{U_{Y;\al}\cap U_{Y;\al'}}
\lr{\cdot,\cdot}_{\al'}\big|_{U_{Y;\al}\cap U_{Y;\al'}}\EE
for all $\al,\al'\!\in\!\cI$.
Furthermore, the overlap maps $h\!\equiv\!\wt{h}_{\al\al'}$ in~\eref{TNIequiv_e} 
with $\Psi_1\!\equiv\!\Psi_{\al}$, $\Psi_2\!\equiv\!\Psi_{\al'}$, 
$W_{\cap}\!\equiv\!\Psi_{\al'}^{-1}(U_{\al})$, and $\bF\!=\!\R$ satisfy
\BE{augatlcond_e9}\begin{aligned}
\big\{\!\wt{h}_{\al\al'}(v)\!\big\}\!\big(\cN^{\fc_1}|_{\pi_{\cN_XY}(v)}\big)
&\subset\cN^{\fc_1}_{\pi_{\cN_XY}(\Psi_{\al}^{-1}(\Psi_{\al'}(v)))}
&~~&\forall~v\!\in\!\Psi_{\al'}^{-1}(U_{\al}),\\
\blr{\!\{\wt{h}_{\al\al'}(v)\!\}(w),\{\wt{h}_{\al\al'}(v)\!\}(w)\!}_{\al}
&=f_{\al\al'}\big(\Psi_{\al'}(v)\!\big)\lr{w,w}_{\al'}
&~~&\forall~v\!\in\!\Psi_{\al'}^{-1}(U_{\al}),\,w\!\in\!\cN^{\fc_1}\big|_{\pi_{\cN_XY}(v)}.
\end{aligned}\EE

\vspace{.18in}

\noindent
By~\eref{augatlcond_e7}, we can choose an inner-product~$\lr{\cdot,\cdot}$ on~$\cN^{\fc_1}$
so that its restriction to $\cN^{\fc_1}|_{U_{Y;\al}}$ is conformally equivalent 
to~$\lr{\cdot,\cdot}_{\al}$ for every~$\al\!\in\!\cI$.
The tubular neighborhood identification~$(W_{\al},\Psi_{\al})$ is then
$\fc_1$-augmented compatible with the $\cN^{\fc_1}$-augmented atlas
$\{\vph_{\al'}\}_{\al'\in\cI}$.
By~\eref{augatlcond_e9}, any two such identifications are $(\cN^{\fc_1},\lr{\cdot,\cdot})$-equivalent.
As in the last paragraph of the proof of Proposition~\ref{Loc2GlRC_prp},
we can thus obtain a finite open cover $\{U_{Y;i}\}_{i\in[N]}$ of~$Y$ and
a collection $\{\!(W_i,\Psi_i)\!\}_{i\in[N]}$ of tubular neighborhood identifications
for $Y\!\subset\!X$ over $U_{Y;i}\!\subset\!Y$ so that 
each $(W_i,\Psi_i)$ is 
$\fc_1$-augmented compatible with $\{\vph_{\al'}\}_{\al'\in\cI}$
and any two of these identifications are $(\cN^{\fc_1},\lr{\cdot,\cdot})$-equivalent.
The claim now follows from Lemma~\ref{Loc2Glaugmerge_lmm} below applied $N\!-\!1$ times.
\end{proof}

\begin{lmm}\label{Loc2Glaugmerge_lmm}
Suppose $X$, $Y$, $\fc$, $\fc_1$, and $\cN^{\fc_1}$ are as in Proposition~\ref{Loc2Glaug_prp},
$\lr{\cdot,\cdot}$ is an inner-product on~$\cN^{\fc_1}$, and
$(W_1,\Psi_1)$ and $(W_2,\Psi_2)$ are tubular neighborhood identifications for $Y\!\subset\!X$
over open subsets $U_1\!\subset\!Y$ and $U_2\!\subset\!Y$, respectively.
If $(W_1,\Psi_1)$ and~$(W_2,\Psi_2)$ are 
$(\cN^{\fc_1},\lr{\cdot,\cdot})$-equivalent, then
there exists a tubular neighborhood identification~$(W,\Psi)$ for~$Y$ over \hbox{$U_1\!\cup\!U_2$} 
which is $(\cN^{\fc_1},\lr{\cdot,\cdot})$-equivalent to~$(W_1,\Psi_1)$ and~$(W_2,\Psi_2)$.
\end{lmm}

\begin{proof} 
We denote by $\pi_{\SO}\!:\SO(\cN^{\fc_1})\!\lra\!Y$ the fiber bundle of orientation-preserving automorphisms
of~$(\cN^{\fc_1},\lr{\cdot,\cdot})$ and~by
$$\pi_{\so}\!:\so(\cN^{\fc_1})\subset T\big(\SO(\cN^{\fc_1})\!\big)\!\big|_{s_{\bI}(Y)}$$
the restriction of the vertical tangent bundle to the section~$s_{\bI}$ of~$\pi_{\SO}$
given by the identity automorphism of each fiber.
Let
$$\exp_{\cN^{\fc_1}}\!:\so(\cN^{\fc_1})\lra \SO(\cN^{\fc_1})$$
be a smooth map so that
$\exp_{\cN^{\fc_1}}\!|_{\so(\cN^{\fc_1})_y}$ is an embedding for every $y\!\in\!Y$,
\begin{gather*}
\pi_{\SO}\!\circ\!\exp_{\cN^{\fc_1}}=\pi_{\SO}\!\circ\!\pi_{\so}\!:\so(\cN^{\fc_1})\lra Y,
\quad \exp_{\cN^{\fc_1}}\!\big|_{s_{\bI}(Y)}=\id_{s_{\bI}(Y)},
\quad\hbox{and}\\
\nd_{s_{\bI}(y)}\exp_{\cN^{\fc_1}}\!=\!\id\!:
\so(\cN^{\fc_1})\!\big|_{s_{\bI}(y)}\lra 
\so(\cN^{\fc_1})\!\big|_{s_{\bI}(y)}\subset 
T_{s_{\bI}(y)}\big(\SO(\cN^{\fc_1})\!\big)~~\forall\,y\!\in\!Y.
\end{gather*}
Let $\cN'\!\subset\!\cN_XY$ be a subbundle complementing~$\cN^{\fc_1}$.\\

\noindent
Continuing with the notation and setup of the proof of Lemma~\ref{Loc2GlRCmerge_lmm},
choose a vector bundle isomorphism~$\wt\exp_Y$ in~\eref{wtexpcond_e0} 
which restricts to an isometry from~$\pi_{TY}^*\cN^{\fc_1}$ to~$\exp^*\cN^{\fc_1}$ 
with respect to~$\lr{\cdot,\cdot}$.
By~\eref{augatlcond_e0}, the section~$h_{\th}$ now satisfies
\begin{equation*}\begin{aligned}
\big\{\!h_{\th}(v)\!\big\}\!\big(\pi_{TY}^*\cN^{\fc_1}\big)&\subset\pi_{TY}^*\cN^{\fc_1}
&\qquad&\forall~v\!\in\!W_{\cap}',\\
\blr{\!\{h_{\th}(v)\!\}(w),\{h_{\th}(v)\!\}(w)\!}&=f(v)\lr{w,w}
&\qquad&\forall~v\!\in\!W_{\cap}',\,w\!\in\!\cN^{\fc_1}\big|_{\pi_{\cN_XY}(v)}.
\end{aligned}\end{equation*}
By~\eref{TNIprop_e2}, the subset
$$W_{\cap}''=\big\{v\!\in\!W_{\cap}'\!:
f(v)^{-1/2}\{h_{\th}(v)\!\}\!\big|_{\pi_{TY}^*\cN^{\fc_1}}\!\in\!
\exp_{\cN^{\fc_1}}\!\!\big(\so(\cN^{\fc_1})_{s_{\bI}(y)}\!\big)\!\big\}$$
is still an open neighborhood of $U_1\!\cap\!U_2$ in~$\cN_XY$.
Define 
$$\wt{h}_{\th}\in\Ga\big(W_{\cap}'',s_{\bI}^*\so(\cN^{\fc_1})\!\big)
\qquad\hbox{by}\quad
\exp_{\cN^{\fc_1}}\!\!\big(\wt{h}_{\th}(v)\!\big)
=f(v)^{-1/2}\{h_{\th}(v)\!\}\!\big|_{\pi_{TY}^*\cN^{\fc_1}}.$$
By~\eref{TNIprop_e2},
\BE{TNIprop_e4}
\wt{h}_{\th}(y)=0\!\equiv\!\big(y,s_{\bI}(y)\!\big)
\qquad\forall~y\!\in\!U_1\!\cap\!U_2.\EE

\vspace{.18in}

\noindent
As in the proof of Lemma~\ref{Loc2GlRCmerge_lmm}, we assume that \hbox{$U_1\!\cup\!U_2\!=\!Y$}
and choose open subsets $U_1',U_2'\!\subset\!Y$ and \hbox{$\eta\!\in\!C^{\i}(Y;\R)$} as before.
Define
\begin{gather*}
h_{\eta} \in\Ga\big(W_{\cap}'',\pi_{\cN_XY}^*\End_{\R}(\cN_XY)\!\big) \qquad\hbox{by}\\
\big\{\!h_{\eta}(v)\!\big\}(w)=\begin{cases}
w\!+\!\eta\big(\pi_{\cN_XY}(v)\!\big)\big(\{h_{\th}(v)\!\}(w)\!-\!w\big),
&\hbox{if}~w\!\in\!\pi_{\cN_XY}^*\cN';\\
f(v)^{\eta(\pi_{\cN_XY}(v)\!)/2}
\big\{\!\exp_{\cN^{\fc_1}}\!\!\big(\eta\big(\pi_{\cN_XY}(v)\!\big)\wt{h}_{\th}(v)\!\big)\!\big\}
(w),&\hbox{if}~w\!\in\!\pi_{\cN_XY}^*\cN^{\fc_1}.
\end{cases}\end{gather*}
In particular,
\begin{gather}\label{hetaprop_e1}
h_{\eta}(v)=\begin{cases}\id,&\hbox{if}~v\!\in\!U_1\!\cap\!U_2;\\
\id,&\hbox{if}~v\!\in\!W_1|_{U_1-\ov{U_2'}}\!\cap\!W_{\cap}'';\\
h_{\th}(v),&\hbox{if}~v\!\in\!W_{\cap}''\!\cap\!W_2|_{U_2-\ov{U_1'}};
\end{cases} \qquad
\big\{h_{\eta}(v)\!\big\}\!\big(\pi_{TY}^*\cN^{\fc_1}\big)\subset\pi_{TY}^*\cN^{\fc_1},\\
\label{hetaprop_e2}
\blr{\!\{h_{\eta}(v)\!\}(w),\{h_{\eta}(v)\!\}(w)\!}
=f(v)^{\eta(\pi_{\cN_XY}(v))/2}\lr{w,w}
~~\forall~w\!\in\!\cN^{\fc_1}\big|_{\pi_{\cN_XY}(v)}.
\end{gather}
The subset $W\!\equiv\!W_1|_{U_1-\ov{U_2'}}\!\cup\!W_{\cap}''\!\cup\!W_2|_{U_2-\ov{U_1'}}$
of~$\cN_XY$ is an open neighborhood of~$Y$ and the map
\begin{gather*}
\Psi\!:W\lra \cN_XY, \\
\Psi(v)=\begin{cases}
\Psi_1(v),&\hbox{if}~v\!\in\!W_1|_{U_1-\ov{U_2'}};\\
\Psi_1\big(\pi_2\big(\wt\exp_Y\big(\eta(\pi_{\cN_XY}(v)\!)\!\big)\th(v),
\big\{\!h_{\eta}(v)\!\big\}(v)\!\big)\!\big),&\hbox{if}~v\!\in\!W_{\cap}''\,;\\
\Psi_2(v),&\hbox{if}~v\!\in\!W_2|_{U_2-\ov{U_1'}};
\end{cases}
\end{gather*}
is well-defined and smooth.\\

\noindent
By~\eref{hetaprop_e1} and~\eref{hetaprop_e2}, $\Psi|_Y\!=\!\id_Y$ and the homomorphism
$$T(\cN_XY)\big|_Y \xlra{\nd\Psi} TX|_Y
\xlra{\pi_Y^{\perp}} \frac{TX|_Y}{TY}\!\equiv\!\cN_XY$$
restricts to the identity homomorphism on \hbox{$\cN_XY\!\subset\!T(\cN_XY)|_Y$}.
By the Inverse Function Theorem, there thus exists a neighborhood $W'\!\subset\!W$
of~$Y$ such that $(W',\Psi|_{W'})$ is a tubular neighborhood identification for $Y\!\subset\!X$.
By the same reasoning as at the end of the proof of  Lemma~\ref{Loc2GlRCmerge_lmm},
$(W',\Psi|_{W'})$ is $(\cN^{\fc_1},\lr{\cdot,\cdot})$-equivalent to~$(W_1,\Psi_1)$ 
and~$(W_2,\Psi_2)$.
\end{proof}

\vspace{.2in}

\noindent
{\it Department of Mathematics, Stony Brook University, Stony Brook, NY 11794\\
azinger@math.stonybrook.edu}

\end{document}